\documentclass[12pt]{amsart}
\usepackage{enumerate,graphicx}
\usepackage{verbatim}

\usepackage{amsmath, amssymb, amsthm}
\usepackage{amsfonts}
\usepackage{hyperref}

\newtheorem{theorem}{Theorem}
\newtheorem{proposition}[theorem]{Proposition}
\newtheorem{lemma}[theorem]{Lemma}
\newtheorem{corollary}[theorem]{Corollary}

\theoremstyle{definition}
\newtheorem{definition}[theorem]{Definition}

\newtheorem{problem}[theorem]{Problem}
\newtheorem*{acknowledgements}{Acknowledgements}

\numberwithin{equation}{section}

\newcommand\Forb{\operatorname{Forb}}
\newcommand\MT{\mathcal{G}_\mathrm{MT}}
\newcommand\SplitMT{\mathcal{G}_\mathrm{SplitMT}}
\newcommand\bard{\overline d}
\newcommand\ourcomment[1]{ \textbf{[#1]} }
\newcommand\oc\ourcomment

\begin{document}

\title{Mock Threshold Graphs}
\author{Richard Behr, Vaidy Sivaraman, and Thomas Zaslavsky}

\address{Department of Mathematical Sciences, Binghamton University.}
\email{\{behr|vaidy|zaslav\}@math.binghamton.edu}

\begin{abstract}
Mock threshold graphs are a simple generalization of threshold graphs that, like threshold graphs, are perfect graphs. Our main theorem is a characterization of mock threshold graphs by forbidden induced subgraphs.  Other theorems characterize mock threshold graphs that are claw-free and that are line graphs.  We also discuss relations with chordality and well-quasi-ordering as well as algorithmic aspects. 
\end{abstract}



\date{December 24, 2016}

\maketitle

\tableofcontents

\section{Introduction}\label{intro}

We define, and study the surprisingly many properties of, a new class of graphs: a simple generalization of threshold graphs that we call ``mock threshold graphs''. 
One reason to study mock threshold graphs is that, like threshold graphs, they are perfect.  We characterize the class of mock threshold graphs by forbidden induced subgraphs; we discuss their properties of chordality, planarity, claw-freeness, and well-quasi-ordering; and we find the line graphs that are mock threshold.  We also treat algorithmic aspects of mock threshold graphs.

A graph $G$ is said to be \emph{threshold} if there are a function $w: V(G) \to \mathbb{R}$ and a real number $t$ such that there is an edge between two distinct vertices $u$ and $v$ if and only if $w(u) + w(v) > t$ (Chv\'{a}tal--Hammer 1973).  The class of threshold graphs has been studied in great detail, mainly because of its simple structure.  There is an entire book about threshold graphs \cite{MP}. A fundamental theorem characterizes the class by forbidden induced subgraphs. 

 \begin{theorem}[Chv\'{a}tal--Hammer 1973]\label{ThresholdFISC}
 A graph is threshold if and only if it contains no induced subgraph isomorphic to $2K_2,P_4$, or $C_4$. 
\end{theorem}

Another fundamental fact about threshold graphs is their characterization by vertex ordering.  
If $G$ is a graph and $X \subseteq V(G)$, then $G{:}X$ denotes the subgraph of $G$ induced on $X$.

\begin{theorem}
A graph $G$ is threshold if and only if $G$ has 
a vertex ordering $v_1, \ldots, v_n$ such that for every $i$ ($1 \leq i \leq n$) the degree of $v_i$ in $G{:}\{v_1, \ldots, v_i\}$ is $0$ or $i-1$. 
\end{theorem}

By relaxing this characterization slightly we get a new, bigger class of graphs.

\begin{definition}
A graph $G$ is said to be \emph{mock threshold}\footnote{With apologies to the Mock Turtle.} if there is a vertex ordering $v_1, \ldots, v_n$ such that for every $i$ ($1 \leq i \leq n$) the degree of $v_i$ in $G{:}\{v_1, \ldots, v_i\}$ is $0,1,i-2$, or $i-1$. 
We write $\MT$ for the class of mock threshold graphs.
\end{definition}

We call such an ordering an \emph{MT-ordering}. Note that a graph can have several MT-orderings. There are several easy but important consequences of the definition.

Although the class of mock threshold graphs is not closed under taking subgraphs, it is hereditary in the sense of induced subgraphs.  

\begin{proposition}\label{CLOSEDUNDERINDUCEDSUBGRAPHS}
 Every induced subgraph of a mock threshold graph is mock threshold.
\end{proposition}

Thus, as with threshold graphs, there exists a characterization by forbidden induced subgraphs, which we describe in Section \ref{forbidden}.

A graph $G$ is \emph{perfect} if $\chi(H) = \omega(H)$ for all induced subgraphs $H$ of $G$ (Berge 1961). 
It is \emph{chordal} if every induced cycle in it is a triangle (Dirac 1961).
It is \emph{split} if its vertex set can be partitioned into a clique and a stable set (F\"{o}ldes--Hammer 1977). 
It is \emph{weakly chordal} if every induced cycle in it or its complement is either a triangle or a square (Hayward 1985).
We have the following chain of inclusions: 
  $$\text{Threshold} \subset \text{Split} \subset \text{Chordal} \subset \text{Weakly Chordal} \subset \text{Perfect}.$$
  All the inclusions except the last one are easy. The last inclusion can be proved directly \cite{RBH} or one can use the Strong Perfect Graph Theorem \cite{CRST} to see the inclusion immediately.  

  Mock threshold graphs are perfect (Proposition \ref{MTGPerfect}) and indeed weakly chordal (Corollary \ref{WEAKLYCHORDAL}) but not necessarily chordal. The cycle of length four is a mock threshold graph that is not chordal.
 
\begin{proposition}\label{CLOSEDUNDERCOMPLEMENTATION}
 The complement of a mock threshold graph is also mock threshold.
\end{proposition}
\begin{proof}
 The same vertex ordering works. 
\end{proof}

\begin{proposition}\label{MTGPerfect}
 A mock threshold graph is perfect.
\end{proposition}
\begin{proof}
 Adding an isolated vertex or a leaf preserves perfection. The Weak Perfect Graph Theorem \cite{LL} tells us that a graph is perfect if and only if its complement is also perfect. Hence the four operations of constructing a mock threshold graph from $K_1$, viz.,  adding an isolated  vertex, adding a leaf, adding a dominating vertex, and adding a vertex that dominates all but one vertex, preserve perfection. Hence every mock threshold graph is perfect.
\end{proof}

\section{Preliminaries}\label{prelim}

All our graphs are finite and simple, that is, we allow neither loops nor multiple edges. 
When we say $G$ contains $H$, we mean that $G$ contains $H$ as a subgraph. 
Let $k$ be a positive integer. Then $P_k$, $C_k$, $K_k$ denote the path, cycle, and complete graph, respectively,  on $k$ vertices.  
We denote the complement of $G$ by $\overline{G}$.  The neighborhood of a vertex $v$ is denoted by $N(v)$. 
For a positive integer $k$, the \emph{$k$-core} of a graph $G$ is the graph obtained from $G$ by repeatedly deleting vertices of degree less than $k$. It is routine to show that this is well defined.  

For $G$ a graph, $\delta(G)$ and $\Delta(G)$ denote the minimum and maximum degree of $G$, respectively.  
The chromatic number of $G$, denoted by $\chi(G)$, is the smallest positive integer $k$ such that its vertices can be colored with $k$ colors so that adjacent vertices receive different colors. The clique number of $G$, denoted by $\omega(G)$, is the largest positive integer $k$ such that the complete graph on $k$ vertices is a subgraph of $G$.  
The \emph{codegree} $\overline{d}(v)$ of a vertex $v$ in $G$ is its degree in the complement $\overline G$.

\section{Simple properties}\label{simple}

We list some easy but useful properties and examples.

\begin{lemma}\label{MINDEGREEMAXDEGREE}
 Let $G$ be a graph on $n$ vertices with $2 \leq \delta(G) \leq \Delta(G) \leq n-3$. Then $G$ is not mock threshold. 
\end{lemma}

A \emph{removable vertex} is a vertex whose degree or codegree is at most 1.  Thus, a graph with no removable vertex cannot be mock threshold.

\begin{lemma}\label{REMOVABLE}
Let $G$ be a graph on $n$ vertices. Let $v \in V(G)$ be removable. Then $G$ is mock threshold if and only if $G - \{v\}$ is mock threshold. 
\end{lemma}
  
\begin{proof}
If $G$ is mock threshold, then every induced subgraph of $G$ is also mock threshold (by Proposition \ref{CLOSEDUNDERINDUCEDSUBGRAPHS}). This proves one direction. For the other direction, assume that $G - \{v\}$ is mock threshold. Thus it has an MT-ordering. Add $v$ as the last vertex to this MT-ordering; since $v$ is removable, this is an MT-ordering for $G$.  Hence $G$ is mock threshold.
\end{proof}

 \begin{proposition}
 $K_n$ and $K_{2,n}$ are mock threshold.
 \end{proposition}

\begin{proposition} \label{FORESTSAREMOCKTHRESHOLD}
 A forest is a mock threshold graph.
\end{proposition}

\begin{proposition}\label{2-CORE}
 A graph is mock threshold if and only if its $2$-core is also mock threshold.
\end{proposition}
\begin{proof}
 One direction follows from Proposition \ref{CLOSEDUNDERINDUCEDSUBGRAPHS}. The other direction follows from the fact that a graph obtained from a mock threshold graph by adding an isolated vertex or a pendant edge is also mock threshold.
\end{proof}

\begin{lemma}\label{DISJOINTCYCLES}
A graph consisting of two disjoint cycles is not mock threshold.
\end{lemma}

\begin{proof}
The graph has no removable vertices.
\end{proof}

\emph{Attaching a graph $H$ to a vertex} $v$ of $G$ means that a vertex of $H$ is identified with $v$.  (It is assumed that $G$ and $H$ are disjoint and that $H$ is not $K_1$.)

\begin{proposition}\label{BIPARTITE}
A bipartite graph is mock threshold if and only if either it is a forest, or it has one component that consists of $K_{2,s}$, $s\geq2$, with a tree optionally attached to each vertex, and any other components are trees.
\end{proposition}

\begin{proof}
It is clear that any such graph has a removable vertex and that deleting that vertex results in a graph of the same type.  

Now, let $G$ be bipartite and mock threshold.  We may prune all leaves and isolated vertices until the only remaining vertices are in cycles.  The resulting graph $G'$ must be mock threshold (Proposition \ref{CLOSEDUNDERCOMPLEMENTATION}) and connected (Lemma \ref{DISJOINTCYCLES}).  Every vertex has degree 2 or more, so there must be a (removable) vertex of codegree 0 or 1.  Let $G'$ have left  vertex set $X$ and right vertex set $Y$ with $r:=|X|$ and $s:=|Y|$; necessarily $r,s \geq 2$.  If $r,s \geq 3$, every vertex has codegree at least 2; therefore $r$, say, equals 2.  Then every vertex $y\in Y$, having degree $d(y)>1$, is adjacent to both members of $X$; that is, $G' = K_{2,s}$.
\end{proof}

The class of mock threshold graphs has the valuable property of being closed under contraction.  Contracting an edge $v_iv_j$ in a simple graph $G$ can be defined as replacing $v_i$ and $v_j$ by a new vertex $v_{ij}$ whose neighborhood is $N(v_i)\cup N(v_j) - \{v_i,v_j\}$.  The notation for the contracted graph is $G/e$.  A \emph{contraction} of $G$ is any graph, including $G$ itself, obtained by a sequence of edge contractions.  An \emph{induced minor} of $G$ is any induced subgraph of a contraction.

\begin{proposition}\label{P:contract}
Any contraction of a mock threshold graph is mock threshold.  In particular, if $G$ is a mock threshold graph and $e$ is an edge in $G$, then $G/e$ is mock threshold.
\end{proposition}

\begin{proof}
Let $v_1,\ldots,v_n$ be an MT-ordering of $G$.  Let $e = v_iv_j$, $i < j$; then $v_j$ cannot be isolated. If $v_j$ is isolated or a leaf in the ordering, then $G/e$ is isomorphic to $G-v_j$.  If $v_j$ is near-dominating or dominating, then deleting $v_i$ and replacing $v_j$ by $v_{ij}$ gives an MT-ordering for $G/e$ in which $v_{ij}$ is also near-dominating or dominating.
\end{proof}

\section{Minimal graphs that are not mock threshold}\label{forbidden}

A graph is a \emph{minimal non-mock threshold} graph (from now on, minimal non-MTG) if it is not mock threshold but every induced proper subgraph is mock threshold.  We write $\Forb(\MT)$ for the class of minimal non-MTGs. It is straightforward to show that a graph is mock threshold if and only if it does not contain any of the minimal non-MTGs as an induced subgraph.

\begin{proposition}\label{HOLESANDANTIHOLES}
 Let $n \geq 5$. Then both $C_n$ and $\overline{C}_n$ are minimal non-MTGs. 
\end{proposition}
\begin{proof}
 By virtue of Proposition \ref{CLOSEDUNDERCOMPLEMENTATION}, it suffices to prove that $C_n$ is a minimal non-mock threshold graph. Since no vertex in $C_n$ is removable, we see, by Lemma \ref{MINDEGREEMAXDEGREE}, that $C_n$ is not a mock threshold graph. Every proper subgraph of $C_n$ is a forest, and hence, by Proposition \ref{FORESTSAREMOCKTHRESHOLD}, a mock threshold graph. 
\end{proof}

\begin{corollary}\label{WEAKLYCHORDAL}
Mock threshold graphs are weakly chordal.
\end{corollary}

\begin{proposition}\label{SMALLGRAPHS}
 Every graph on at most five vertices except $C_5$ is mock threshold.
 \end{proposition}
 \begin{proof}
  As shown in Proposition \ref{HOLESANDANTIHOLES}, $C_5$ is not mock threshold. Clearly, every graph on at most four vertices is mock threshold. This implies that a graph on $5$ vertices is mock threshold if it has a vertex whose degree is not $2$. Every graph on $5$ vertices has such a vertex with the exception of $C_5$. 
 \end{proof}

We are interested in determining $\Forb(\MT)$. Our main theorem is the following: 

\begin{theorem}[Forbidden Induced Subgraphs of Mock Threshold Graphs]\label{FORBIDDEN}
A graph is mock threshold if and only if it contains none of the following as an induced subgraph:
\begin{enumerate}[{\rm(a)}]
\item Cycles of length at least $5$ and their complements.
\item The $318$ graphs, each with at most $10$ vertices, that are the non-cyclic graphs in Figures \ref{ForbiddenMTGraphs7Vertices}--\ref{ForbiddenMTGraphs10Vertices} and their complements.
\end{enumerate}
\end{theorem}
 
To prove the theorem we will need the following lemmas.  Let $G$ be a minimal non-MTG on $n$ vertices that is neither a cycle nor its complement. By Proposition \ref{SMALLGRAPHS}, $n > 5$.

\begin{lemma} \label{DEGREEBOUND}
The minimum degree of $G$ is at least $2$.  The maximum degree of $G$ is at most $n-3$.
\end{lemma}

\begin{proof}
Suppose $G$ has a vertex $v$ of degree $0,1,n-2,$ or $n-1$. Since $G$ is a minimal non-MTG, $G - v$ is mock threshold, and hence has an MT-ordering. Adding $v$ as the last vertex in the MT-ordering, we get an MT-ordering for G, a contradiction. 
\end{proof}

We call a vertex in a graph co-divalent if it is non-adjacent to exactly two vertices. 

\begin{lemma}\label{DIVALENTORCODIVALENT}
Every vertex in $G$ is either adjacent to a divalent vertex or non-adjacent to a co-divalent vertex.
\end{lemma}

\begin{proof}
 Let $v$ be a vertex in $G$. Since $G$ is a minimal non-MTG, $G - v$ is mock threshold, and hence has a vertex $w$ that has degree $0,1,n-3$, or $n-2$. By Lemma \ref{DEGREEBOUND}, $w$ has degree at least $2$ and at most $n-3$ in $G$, and hence at least $1$ and at most $n-3$ in $G - v$. If $w$ has degree $1$ in $G$, then $v$ must be adjacent to $w$, since by the Lemma \ref{DEGREEBOUND} $w$ must have degree at least $2$ in $G$. Hence $v$ is adjacent to the divalent vertex $w$. If $w$ has degree $n-3$, then since the degree of $w$ in $G$ is at most $n-3$, $v$ is non-adjacent to $w$, and the degree of $w$ in $G - v$ is $n-3$. Hence $v$ is non-adjacent to the co-divalent vertex $w$.   
\end{proof}

\begin{lemma}\label{3OR4CYCLE}
Every vertex in $G$ is in an induced cycle of length $3$ or $4$.
\end{lemma}
\begin{proof}
Let $v$ be a vertex in $G$. By Lemma \ref{DEGREEBOUND}, $G$ has minimum degree at least $2$, and hence has a cycle passing through $v$. Let $C$ be a smallest cycle passing through $v$. Then $C$ is an induced subgraph of $G$. Since $G$ is a minimal non-MTG and $n > 5$, $G$ does not contain an induced cycle of length at least $5$. Thus $C$ has length $3$ or $4$. 
\end{proof}

\begin{lemma}\label{ABSENCEOFTRIPLET}
There do not exist three divalent vertices with the same neighborhood.
\end{lemma}
\begin{proof}
Let $x,y,z$ be divalent vertices with the same neighborhood $\{a,b\}$. Since $G$ is a minimal non-MTG, $G-z$ is mock threshold, and hence has an MT-ordering. Without loss of generality we may assume $a$ was above $b$, and $x$ was above $y$ in the ordering. 
Suppose $a$ was above $x$ in the MT-ordering of $G-z$. Then placing $z$ immediately above $x$ gives an MT-ordering of $G$. Suppose $x$ was above $a$ in the MT-ordering of $G-z$. Then placing $z$ as the first vertex followed by the MT-ordering of $G-v$ gives an MT-ordering of $G$. In either case, we get an MT-ordering of $G$, a contradiction. 
\end{proof}

If the neighbors of a divalent vertex are adjacent, we say that the divalent vertex is of triangle-type. If the neighbors of a divalent vertex are non-adjacent, we say that the divalent vertex is of seagull-type.

\begin{lemma}\label{COMMONNEIGHBOR}
The neighbors of a seagull-type divalent vertex have another common neighbor.
\end{lemma}
\begin{proof}
Let $v$ be a divalent vertex in $G$ of seagull type. By Lemma \ref{3OR4CYCLE}, $v$ is a cycle of length $3$ or $4$. Since $v$ is of seagull type, it cannot be in a $3$-cycle. Hence $v$ is in a $4$-cycle i.e., the two neighbors of $v$ have another common neighbor.
\end{proof}

\begin{lemma}\label{FULLSET}
If $X$ is a set of vertices in $G$ such that every vertex in $X$ has at least two neighbors in $X$ and at least two non-neighbors in $X$, then $X = V(G)$.
\end{lemma}
\begin{proof}
By Lemma \ref{MINDEGREEMAXDEGREE}, the induced subgraph $G{:}X$ is not mock threshold. By the minimality of $G$, $X = V(G)$. 
\end{proof}

A vertex set satisfying the condition in Lemma \ref{FULLSET} is called a \emph{full set}. The reason is that in a minimal non-MTG, the vertex set is the only full set.

\begin{lemma}\label{DIVALENTVERTICES}
The total number of divalent and co-divalent vertices in $G$ is at least $\frac{n}{2}$.
\end{lemma}
\begin{proof}
Lemma \ref{DIVALENTORCODIVALENT} says that every vertex in a minimal non-MTG is adjacent to a divalent vertex or non-adjacent to a co-divalent vertex. Hence the number of vertices in $G$ is at most twice the number of vertices that are either divalent or co-divalent.
\end{proof}

\begin{lemma}\label{NOTCONNECTED}
If $G$ is not connected, then $G$ has at most $8$ vertices. 
\end{lemma}
\begin{proof}
Let $H_1$ and $H_2$ be two components of $G$. Since the minimum degree of a vertex in $G$ is at least $2$, $G$ has a cycle in each component. Let $C_i$ be a smallest cycle (and therefore chordless) in $H_i$ ($i=1,2$). Note that $|C_i| \in \{3,4\}$ since a cycle of length at least $5$ is a minimal non-MTG. 
Now $C_1 \cup C_2$ is a non-MTG, and therefore $G = C_1 \cup C_2$, and hence $G$ has at most $8$ vertices. 
\end{proof}

\begin{lemma} \label{ISTHMUS}
If $G$ has an isthmus, then $G$ has at most $8$ vertices. 
\end{lemma}
\begin{proof}
If $G$ is not connected then we are done by Lemma \ref{NOTCONNECTED}. Hence we may assume that $G$ is connected. Let $e=vw$ be an isthmus in $G$. Let the two components in $G - e$ be $H_1$ and $H_2$ where $v \in V(H_1)$ and $w \in V(H_2)$. Since every vertex in $G$ has minimum degree at least $2$, $H_1$ and $H_2$ each has a cycle. Let $C_i$ be a shortest cycle in $H_i$ ($i = 1,2$); then $|C_i| \in \{3,4\}$.

If $v \in V(C_1)$ and $w \in V(C_2)$, then $C_1 \cup C_2 + e$ is a non-MTG and hence is equal to $G$.  
Otherwise, $C_1 \cup C_2$ is a non-MTG and hence is equal to $G$.  
In both cases, $G$ has at most $8$ vertices. 
\end{proof}

\begin{lemma}\label{CUTVERTEX}
If $G$ has a cutvertex, then $G$ has at most $9$ vertices. 
\end{lemma}
\begin{proof}
If $G$ is connected or has an isthmus, we are done by Lemmas \ref{NOTCONNECTED} and \ref{ISTHMUS}. Hence we may assume that $G$ is connected and has no isthmuses. Let $v$ be a cutvertex in $G$. Let the components of $G-v$ be $H_1, \ldots, H_k$ ($k \geq 2$). Since $v$ has at least two non-neigbors, we have the following two cases:

\emph{Case 1: $H_i - N(v) = \emptyset$ for $2 \leq i \leq k$.}  
Let $a,b$ be adjacent vertices in $H_2$. (Such vertices must exist because $G$ has no isthmi.) 
If $H_1 - N(v)$ contains an induced cycle $C$, then $V(C) \cup \{v,a,b\}$ is a full set. This is because every vertex is in a cycle and vertices in C are non-adjacent to both $a$ and $b$.
Hence we may assume that $H_1 - N(v)$ is a forest. If there is a tree $T$ with an edge, then it has at least two leaves. Choose two leaves $y,z$ such that the distance between them is minimum, and let $P$ be the unique path from $y$ to $z$. Since $G$ does not contain an induced cycle of length at least $5$, $P$ has length at most $2$. Let $y_1$ be a neighbor of  $y$, and let $z_1$ be a neighbor of $z$ such that $y_1,z_1 \in N(v)$. Now $V(P) \cup \{y_1,z_1,v,a,b\}$ is a full set. This is because every vertex has at least two neighbors and every vertex in P is non-adjacent to both $a$ and $b$. Hence we may assume that $H_1 - N(v)$ is an edgeless forest. Let $y,z$ be vertices in the forest. The set $\{v,a,b, y,z\}$ together with any two neighbors of $y$ and any two neighbors of $z$ is a full set because every vertex has at least two neighbors and the vertices $y,z$ and their two neighbors are non-adjacent to both $a$ and $b$.

\emph{Case 2: $H_i - N(v) \not = \emptyset$ for $i=1,2$.}  
Let $x$ be a vertex in $H_1$ that is non-adjacent to $v$. Let $y$ be a vertex in $H_2$ that is non-adjacent to $v$. Let $C_1$ be a smallest cycle passing through $x$, and let $C_2$ be a smallest cycle passing through $y$. We claim that $X = V(C_1) \cup V(C_2)$ is a full set in $G$.  Every vertex in $G{:}X$ is in a cycle and hence at least two neighbors. Also every vertex except $v$ has at least two non-neighbors because the two cycles were chosen in different components. But the vertex $v$ is not adjacent to either $x$ or $y$. This proves that $X$ is a full set irrespective of whether or not $v \in X$.

In all cases, we have established a full set $X$ of size at most $9$. By Lemma \ref{FULLSET}, $V(G) = X$. This completes the proof. 
\end{proof}

We will use the fact that every vertex in $G$ has at least two neighbors and at least two non-neighbors (Lemma \ref{DEGREEBOUND}) in the sequel without mention. 
So when we make a statement there exists a vertex that is a neighbor (or non-neighbor) of some vertex, it means we are using this fact. To show $G$ has size at most $k$, it suffices (by Lemma \ref{FULLSET}) to show the existence of a full set of size at most $k$.  In other words, the following six lemmas use Lemma \ref{DEGREEBOUND} and Lemma \ref{FULLSET}, but we will not mention them explicitly.

\begin{lemma}
If $G$ has two seagull-type divalent vertices with no common neighbor, then $G$ has at most $8$ vertices.
\end{lemma}
\begin{proof}
Let $a$ and $b$ be seagull-type divalent vertices. Let $c,d$ be the neighbors of $a$, and $e,f$ be the neighbors of $b$ with $\{c,d\} \cap \{e,f\} = \emptyset$. By Lemma \ref{COMMONNEIGHBOR}, there exists a vertex $g$ adjacent to both $c$ and $d$, and there exists a vertex $h$ that is adjacent to both $e$ and $f$. Note that $g$ and $h$ need not be distinct. Now we claim that $X = \{a,b,c,d,e,f,g,h\}$ is a full set in $G$. All we need to show that every vertex in $X$ has at least two neighbors and at least two non-neighbors in $X$. The vertices $a$,$b$, being divalent clearly satisfy the condition of having at least two neighbors and at least two non-neighbors. The vertex $c$ has neighbors $a,g$ and non-neighbors $b,d$. The vertex $d$ has neighbors $a,g$ and non-neighbors $c,b$. The vertex $e$ has neighbors $b,h$ and non-neighbors $f,a$. The vertex $f$ has neighbors $b,h$ and non-neighbors $e,a$. The vertex $g$ has neighbors $c,d$ and non-neighbors $a,b$. The vertex $h$ has neighbors $e,f$ and non-neighbors 
$a,b$. Hence $X$ is a full set. 
\end{proof}

From now on, we will not give detailed reasons why a set of vertices is a full set. They are all similar to the one above, and the reader can verify them easily. The presence of two divalent vertices makes the verification easy.

\begin{lemma}
If $G$ has two seagull-type divalent vertices with exactly one common neighbor, then $G$ has at most $7$ vertices.
\end{lemma}
\begin{proof}
Let $a$ and $b$ be seagull-type divalent vertices. Let $c,d$ be the neighbors of $a$, and $d,e$ be the neighbors of $b$ where $c \not = e$. 
If $ce$ is an edge, then $\{a,b,c,d,e\}$ is a full set, and we are done. Hence we may assume that $ce$ is not an edge. By Lemma \ref{COMMONNEIGHBOR}, there exist vertices $f$ and $g$ such that $f$ is adjacent to $c$ and $d$, and $g$ is adjacent to $d$ and $e$. Now $\{a,b,c,d,e,f,g\}$ is a full set, and we are done.   
\end{proof}

\begin{lemma}
If $G$ has two triangle-type divalent vertices with no common neighbor,
then $G$ has at most $10$ vertices.
\end{lemma}
\begin{proof}
Let $a$ and $b$ be triangle-type divalent vertices. Let $c,d$ be the neighbors of $a$, and $e,f$ be the neighbors of $b$ with $\{c,d\} \cap \{e,f\} = \emptyset$. Suppose there is a vertex $z$  that is adjacent to none of $c,d,e,f$. Let $C$ be a smallest cycle passing through $z$. We know that $C$ has length $3$ or $4$. Now $\{a,b,c,d,e,f\} \cup V(C)$ is a full set, and we are done. Hence we may assume that every vertex outside $\{a,b,c,d,e,f\}$ has at least one neighbor in $\{c,d,e,f\}$. 

Suppose there is a vertex $x$ with exactly one neighbor in $\{c,d,e,f\}$, say $e$. Let $y$ be a neighbor of $x$ other than $e$. Let $u$ be a vertex non-adjacent to $e$ other than $a$. Let $v$ be a neighbor of $u$ outside $\{c,d,e,f\}$, and if no such vertex exists, then let $v$ be any neighbor of $u$. Then $\{a,b,c,d,e,f,u,v,x,y\}$ is a full set, and we are done. Hence we may assume that every vertex outside $\{a,b,c,d,e,f\}$ is adjacent to at least two vertices in $\{c,d,e,f\}$. 

Let $c',d',e',f'$ be vertices non-adjacent to 
$c,d,e,f$, 
respectively, such that $c' \not = b$, $d' \not = b$, $e' \not = a$, and $f' \not = a$. Then $\{a,b,c,d,e,f,c',d',e',f'\}$ is a full set, and we are done.   
\end{proof}

\begin{lemma}
If $G$ has two triangle-type divalent vertices with exactly one common neighbor,
then $G$ has at most $10$ vertices.
\end{lemma}
\begin{proof}
Let $a,b$ be triangle-type divalent vertices. Let $c,d$ be the neighbors of $a$, and $d,e$ be the neighbors of $b$ where $c \not = e$. 
 
Suppose there exists a co-divalent vertex $z$ outside $\{a,b,c,d,e\}$. Suppose there exists a vertex $f$ non-adjacent to all three of $c,d,e$. Let $f'$ be a neighbor of $f$ other than $z$. Let $d'$ be a non-neighbor of $d$ other than $f$, and $d''$ be a neighbor of $d'$ other than $z$. Now $\{a,b,c,d,e,f,f',d',d'',z\}$ is a full set. Hence we may assume every vertex outside $\{a,b,c,d,e\}$ has at least one neighbor in $\{c,d,e\}$. Let $c'$ be a non-neighbor of $c$ such that  $c' \not = b$. Let $d',d''$ be two non-neighbors of $d$. Let $e'$ be a non-neighbor of $e$ such that $e' \not = a$. Now $\{a,b,c,d,e,c',d',d'',e',z\}$ is a full set. 
Hence we may assume there is no co-divalent vertex outside $\{a,b,c,d,e\}$.

Suppose $c,e$ are both co-divalent. Suppose $ce$ is not an edge. Let $d',d''$ be two non-neighbors of $d$. Then $\{a,b,c,d,e,d',d''\}$ is a full set. Hence we may assume $ce$ is an edge. 
Suppose there exists a vertex $f$ non-adjacent to all three of $c,d,e$. Let $f', f''$ be neighbors of $f$. Let $d'$ be a non-neighbor of $d$ other than $f$. Now $\{a,b,c,d,e,f,f',f'',d'\}$ is a full set. Hence we may assume that every vertex outside $\{a,b,c,d,e\}$ has at least one neighbor in $\{c,d,e\}$. 
Let $c'$ be a non-neighbor of $c$ and $e'$ be a non-neighbor of $e$ such that $c' \not = b$ and $e' \not = a$. Suppose $c'=e'$. Let $c''$ be a neighbor of $c'$. Let $d',d''$ be two non-neighbors of $d$. Then $\{a,b,c,d,e,c',c'',d',d''\}$ is a full set. Hence we may assume $c' \not = e'$. Then $ec'$ and $ce'$ are edges. 
Suppose $d$ is adjacent to both $c'$ and $e'$. Let $d',d''$ be two non-neighbors of $d$. Now $\{a,b,c,d,e,c',e',d',d''\}$ is a full set. Hence we may assume that $d$ is non-adjacent to at least one of $c',e'$. 
Suppose $d$ is non-adjacent to both $c'$ and $e'$. Let $c''$ be a neighbor of $c'$ and $e''$ be a neighbor of $e'$, such that $c'' \not =e$ and $e'' \not = c$. Then $\{a,b,c,d,e,c',e',c'',e''\}$ is a full set. 
Hence we may assume that $d$ is adjacent to exactly one of $c',e'$, say $c'$. Let $e''$ be a neighbor of $e$ and $d'$ be a non-neighbor of $d$, such that $d' \not = e'$. Now $\{a,b,c,d,e,c',e',e'',d'\}$ is a full set. 
Hence we may assume $c,e$ are not both co-divalent.

By Lemma \ref{DIVALENTORCODIVALENT} applied to vertices $a$ and $b$, one of $c,e$ must be divalent, but then $d$ would be a cutvertex, and we are done by Lemma \ref{CUTVERTEX}. 
\end{proof}

\begin{lemma}
If $G$ has two divalent vertices, one of triangle-type and the other seagull-type, with no common neighbor,
then $G$ has at most $10$ vertices.
\end{lemma}

\begin{proof}
Let $a$ be a triangle-type divalent vertex with neighbors $c,d$. Let $b$ be a seagull-type divalent vertex with neighbors $e,f$ such that $\{c,d\} \cap \{e,f\} = \emptyset$. Let $g$ be a vertex adjacent to both $e$ and $f$. Let $m$ be the number of edges with one endpoint in $\{c,d\}$ and the other endpoint in $\{e,f\}$. Suppose $m=0$ or $1$. Then $\{a,b,c,d,e,f,g\}$ is a full set and we are done. Suppose $m=2$. If $ed$ and $cf$ are edges, or if $ec$ and $df$ are edges, then $\{b,c,d,e,f\}$ is a full set. If $ec$ and $ed$ are edges, or if $fc$ and $fd$ are edges, then $\{a,b,c,d,e,f,g\}$ is a full set. If $ce$ and $cf$ are edges, or if $de$ and $df$ are edges, then $\{b,c,d,e,f\} \cup V(C)$,  where $C$ is a smallest cycle passing through a non-neighbor of $c$ other than $b$, is a full set. Suppose $m=3$. Say $ce$ is not an edge. Let $d'$ be a vertex not adjacent to $d$ such that $d' \not = b$. Let $C$ be a smallest cycle passing through $d'$. Then $\{a,b,c,d,e,f\} \cup V(C)$ is a full set. Suppose $m=4$. If 
there is a vertex $z$ non-adjacent to both of $c,d$, then $\{a,b,c,d,e,f\} \cup V(C)$, where $C$ is a smallest cycle passing through $z$, is a full set. Hence we may assume every vertex outside $\{a,b,c,d,e,f\}$ is adjacent to at 
least one of $c,d$. Let $c'$ be a non-neighbor of $c$, $c''$ be a neighbor of $c'$, $d'$ be a non-neighbor of $d$, and $d''$ be a neighbor of $d'$ such that $c' \not = b, d' \not = b, c'' \not = d,$ and  $d'' \not = c$. Then $\{a,b,c,d,e,f,c',c'',d',d''\}$ is a full set. In all cases, we have established the existence of a full set of size at most $10$. The proof is complete. 
\end{proof}

\begin{lemma}
If $G$ has two divalent vertices, one of triangle-type and the other seagull-type, with exactly one common neighbor,
then $G$ has at most $10$ vertices.
\end{lemma}
\begin{proof}
Let $a$ be a triangle-type divalent vertex with neighbors $c,d$. Let $b$ be a seagull-type divalent vertex with neighbors $d,e$ such that $c \not = e$. Suppose $ce$ is an edge. If there exists a vertex $f$ non-adjacent to both $c$ and $d$, then $\{a,b,c,d,e\} \cup V(C)$, where $C$ is a shortest cycle passing through $f$, is a full set. Hence we may assume that every vertex outside $\{a,b,c,d,e\}$ has at least one of $c,d$ as neighbor. Let $c'$ be a non-neighbor of $c$, and $c''$ be a neighbor of $c'$ such that $c' \not = b$ and $c'' \not = d$. Let $d'$ be a non-neighbor of $d$, and $d''$ be a neighbor of $d'$ such that $d' \not = e$ and $d'' \not = c$. Now $\{a,b,c,d,e,c',c'',d',d''\}$ is a full set. 
 
Hence we may assume that $ce$ is not an edge. If there is a vertex $f$ adjacent to $e$ but not to $d$, then $\{a,b,c,d,e\} \cup V(C)$, where $C$ is a smallest cycle passing through $f$ is a full set. Hence we may assume that every vertex that is adjacent to $e$ is also adjacent to $d$. Let $e'$ be a vertex adjacent to $e$, and $d'$ be a vertex non-adjacent to $d$ such that $e' \not = b$ and $d' \not = e$. Now $\{a,b,c,d,e,e'\} \cup V(C)$, where $C$ is a smallest cycle passing through $d'$, is a full set. In all cases, we have established the existence of a full set of size at most $10$. The proof is complete. 
\end{proof}

\begin{proof}[\textbf{Proof of the main theorem}]
Let $G$ be a minimal non-MTG that is neither a cycle of length at least $5$ nor its complement. Suppose $n \geq 10$. By Proposition \ref{DIVALENTVERTICES}, $G$ has at least $5$ vertices that are either divalent or co-divalent. By going to the complement if necessary, we may assume that $G$ has at least $3$ divalent vertices. By Lemma \ref{ABSENCEOFTRIPLET}, two of them must have distinct neighborhoods. We have six cases depending on the type of these two divalent vertices (triangle-type or seagull-type) and whether they have $0$ or $1$ common neighbors. The previous six lemmas tell us that in each case $G$ has at most $10$ vertices. Hence $G$ has at most $10$ vertices. 

The complete list of non-cycle members of $\Forb(\MT)$ was generated using a computer program created in Wolfram Mathematica 10. The program takes a list of all graphs with at most $10$ vertices as input, and for each graph in the list determines whether or not it is a minimal non-MTG. First, it tests whether a graph $G$ is mock threshold by examining the degrees of its vertices and iteratively removing removable vertices. If it succeeds in removing all the vertices, then $G$ is mock threshold, as we have discovered an MT-ordering. Such a graph is discarded. If $G$ is not mock threshold, we generate all of its vertex-deleted subgraphs, and test each one of these in turn to see if it is mock threshold or not. If all vertex-deleted subgraphs of $G$ are mock threshold, we know that $G$ is a minimal non-MTG.
There are $318$ of them. They are shown in Figures \ref{ForbiddenMTGraphs7Vertices}, \ref{ForbiddenMTGraphs10Vertices}, \ref{ForbiddenMTGraphs8Vertices}, \ref{ForbiddenMTGraphs9VerticesPart1}, and \ref{ForbiddenMTGraphs9VerticesPart2}. This establishes the theorem. 
\end{proof}

Our main theorem can be restated as a criterion on weakly chordal graphs:

\begin{corollary}
 A weakly chordal graph is mock threshold if and only if it contains as an induced subgraph none of the $318$ forbidden induced subgraphs of Theorem \ref{FORBIDDEN}(b).
\end{corollary}

The following observation singles out a structure common to all minimal non-MTGs on $10$ vertices or their complements, with the exception of the $10$-cycle.

\begin{lemma}[Butterfly Lemma]\label{L:butterfly}
Let $G$ be a $10$-vertex graph with vertices $x_i,y_i,z_i$, where $ 1 \leq i \leq 4$ and $z_1=z_2$ and $z_3=z_4$,  with edges $x_iy_i$, $x_iz_i$, $y_iz_i$, $y_iz_{i+2}$, $z_1z_3$, and where the $x_i$ are divalent. 
There may also be any edges $y_iy_j$ (see Figure \ref{F:butterfly}).
Then $G$ is a minimal non-MTG.
\end{lemma}

\begin{figure}[hbt]
\includegraphics[scale=.875]{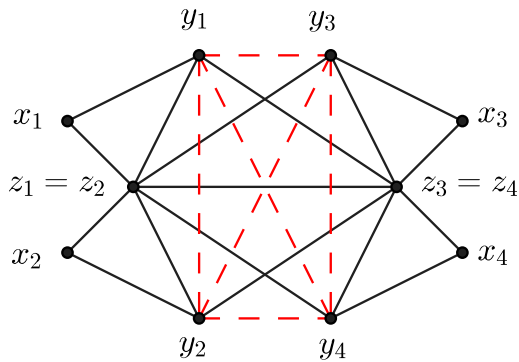}
\caption{The butterfly skeleton (solid lines), common to all the graphs of Lemma \ref{L:butterfly}, and the optional edges (dashed lines) that can be added at will to generate all forbidden induced subgraphs of order $10$ other than $C_{10}$ and $\overline C_{10}$.  For all the non-isomorphic graphs see Figure \ref{ForbiddenMTGraphs10Vertices}.}
\label{F:butterfly}
\end{figure}

\begin{proof}
Every vertex in $G$ has at least two neighbors and at least two non-neighbors, so by Lemma \ref{MINDEGREEMAXDEGREE}, $G$ is not mock threshold. Consider $G-x_1$. The following is a sequence of removable vertices: $z_3,x_3,x_4,y_3,y_4,y_1,z_1,x_2$.  In both $G-y_1$ and $G-z_1$, the vertex $x_1$ becomes removable and we then have the previous sequence of removable vertices. By symmetry, deleting any vertex from $G$ results in a mock threshold graph.  Hence $G$ is a minimal non-MTG.
\end{proof}

Since a contracted mock threshold graph is still mock threshold (Proposition \ref{P:contract}), the list of forbidden subgraphs can be reduced by keeping only those that are minimally non-MTG under contraction as well as taking induced subgraphs.  Notably, only one cycle is needed:  $C_5$ is the only necessary exclusion.  That does not apply to cycle complements; they are all minimally non-MTG under contraction; we overcome this difficulty in Theorem \ref{T:contract-induced-min} by combining a graph and its complement.  
By inspecting the other 318 forbidden induced subgraphs we found those that are minimally non-MTG under both operations.

\begin{theorem}[Forbidden Graphs under Induction and Contraction]\label{T:contract-induced-min}
The number of forbidden induced subgraphs other than $C_5$ and cycle complements that are minimally non-mock threshold under contraction is $241$.  They are the graphs of Figures \ref{ForbiddenMTGraphs7Vertices}--\ref{ForbiddenMTGraphs10Vertices}, and their complements, that are marked as contraction minimal.

A graph is mock threshold if and only if it and its complement have none of the following as a contraction of an induced subgraph, or equivalently as an induced subgraph of a contraction:
\begin{enumerate}[{\rm(a)}]
\item $C_5$.
\item The forbidden induced subgraphs shown in Figures \ref{ForbiddenMTGraphs7Vertices}--\ref{ForbiddenMTGraphs10Vertices} that are marked as minimal under contraction.
\end{enumerate}
\end{theorem}

\begin{proof}
A graph is minimally non-mock threshold under contraction and taking induced subgraphs if and only if it is in $\Forb(\MT)$ and every edge contraction is mock threshold.  We carried out a computer search for this property with the result listed in Figures \ref{ForbiddenMTGraphs7Vertices}-\ref{ForbiddenMTGraphs10Vertices}.

The equivalence of the two criteria is proved by considering a non-MTG $G$.  By Theorem \ref{FORBIDDEN}, $G$ or $\overline G$ has an induced subgraph $H$ that is $C_k$ for some $k\geq5$ or is one of those shown in the figures.  If $H$ is contraction minimal and we are done.  If $H$ is not contraction minimal, some contraction $H/S$ ($S$ is the set of contracted edges) has an induced subgraph $H'$ that is forbidden and contraction minimal.  Then $G/S$ (or $\overline G/S$) has $H/S$ as an induced subgraph, of which $H'$ is an induced subgraph, and we are done.
\end{proof}

\section{Claw-freeness}\label{clawfree}

A \emph{claw} is an induced subgraph that is isomorphic to $K_{1,3}$; its trivalent vertex is called its \emph{center}.  A graph is \emph{claw-free} if it contains no claw.  Claw-free graphs are interesting for several reasons, initially as a generalization of line graphs and later because of good algorithmic properties (see the survey \cite{CFGS}).

To begin with we describe the claw-free threshold graphs.  The characterization is easy to prove from a threshold ordering.

\begin{proposition}\label{P:clawfreethresh}
A graph is claw-free threshold if and only if it consists of isolated vertices and possibly one component that has a vertex whose deletion results in a complete graph.
\end{proposition}

We can explicitly describe all mock threshold graphs that are claw-free.  
\emph{Hanging a path off a vertex} $v$ of $G$ means that the path is attached to $G$ by identifying $v$ with an endpoint of the path; we assume that any such path has positive length.

\begin{theorem}\label{T:clawfree}
A graph $G$ is a claw-free mock threshold graph if and only if either every component is a path, or else just one component $G_1$ is not a path, $G_1$ consists of a $2$-core $G_2$ and at most one path hanging off each vertex of $G_2$ whose $G_2$-neighborhood is a clique, and the complement of $G_2$ is one of the following types {\rm I--\ref{T:clawfreeLAST}}.
\begin{enumerate}[{\rm \quad I.}]
\item\label{T:clawfreeF} $\overline G_2$ is a forest of order at least $3$ in which every vertex has at least $2$ non-neighbors. 
\item\label{T:clawfreeB} $\overline G_2$ has a component that consists of $K_{2,s}$, $s\geq2$, with a tree optionally attached to each vertex; all other components are trees; and every vertex has at least $2$ non-neighbors. 
\item\label{T:clawfree3} $\overline G_2$ is a triangle with one or more pendant edges attached to each of at least two vertices. 
\item\label{T:clawfree1} $\overline G_2$ consists of $K_{2,p}$ with $p\geq2$, whose vertex classes are $X=\{x_1,x_2\}$ and $Z$ of order $p$; also $Y$ of order $s\geq2$ of which all elements are adjacent to $x_1$ and at least one is adjacent to $x_2$; also $w$ adjacent to every vertex in $X\cup Y$; and at least one pendant edge incident to $w$ (Figure \ref{F:clawfree1}).  
\item\label{T:clawfree20} $\overline G_2$ is $K_{1,1,r}$ with $r\geq2$, whose vertex classes are $\{v\}$, $\{w\}$, and $X$; and $\alpha$ and $\beta$ pendant edges incident to $v$ and $w$, respectively, where $\alpha,\beta\geq2$ (Figure \ref{F:clawfree2}).  
\item\label{T:clawfree20p} $\overline G_2$ consists of $K_{1,1,r}$ with $r\geq2$, whose vertex classes are $\{v\}$, $\{w\}$, and $X$; $K_{2,p}$ with $p>0$, whose vertex classes are $\{w,z\}$ and $P$; the edge $vz$; and $\alpha$ and $\beta$ pendant edges incident to $v$ and $w$, respectively, where $\alpha\geq1$ and $\beta\geq0$, but $\beta\geq1$ if $p=1$ (Figure \ref{F:clawfree2}).  
\item\label{T:clawfree21} $\overline G_2$ consists of $K_{1,1,r}$ with $r\geq2$, whose vertex classes are $\{v\}$, $\{w\}$, and $X$; $K_{1,r}$ whose vertex classes are $\{y\}$ and $X$; and $\alpha$ and $\beta$ pendant edges incident to $v$ and $w$, respectively, where $\alpha,\beta\geq1$ (Figure \ref{F:clawfree2}).  
\item\label{T:clawfree21q} $\overline G_2$ consists of $K_{1,1,r}$ with $r\geq2$, whose vertex classes are $\{v\}$, $\{w\}$, and $X$; $K_{1,r}$ whose vertex classes are $\{y\}$ and $X$; $K_{2,q}$ with $p>0$, whose vertex classes are $\{w,y\}$ and $Q$; and $\alpha$ and $\beta$ pendant edges incident to $v$ and $w$, respectively, where $\alpha\geq1$ and $\beta\geq0$ (Figure \ref{F:clawfree2}).  
\item\label{T:clawfree22} $\overline G_2$ consists of $K_{1,1,r}$ with $r\geq2$, whose vertex classes are $\{v\}$, $\{w\}$, and $X$; $K_{r,s}$ with $s\geq1$, whose vertex classes are $X$ and $Y$; and $\alpha$ and $\beta$ pendant edges incident to $v$ and $w$, respectively, where $\alpha+\beta>0$ if $r=2$ (Figure \ref{F:clawfree2}).  
\label{T:clawfreeLAST}
\end{enumerate}
\end{theorem}

Note that in Type \ref{T:clawfreeF} the forbidden forests are a star $S_k$, a disjoint union $S_k \cup K_1$, and $S_k$ with a pendant edge attached to a non-central vertex.

\begin{figure}[hbt]
\includegraphics[scale=.5]{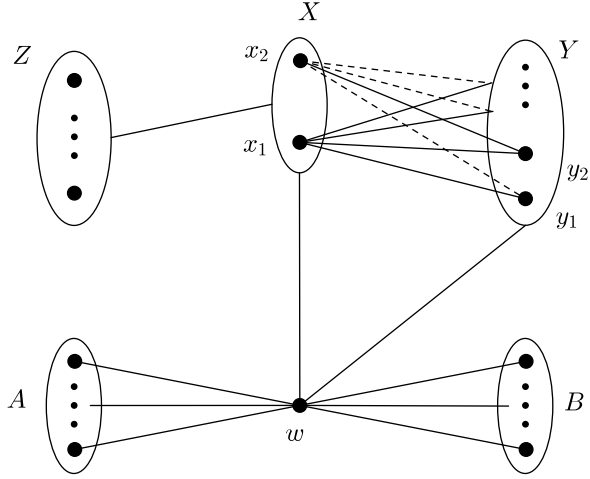}
\caption{A graph whose complement is claw-free, that has a triangle and has a unique vertex that belongs to every triangle ($t=1$).  If $B=\emptyset$, this is a mock threshold graph.  The dashed lines represent edges that may or may not exist.}
\label{F:clawfree1}
\end{figure}

\begin{figure}[hbt]
\includegraphics[scale=.375]{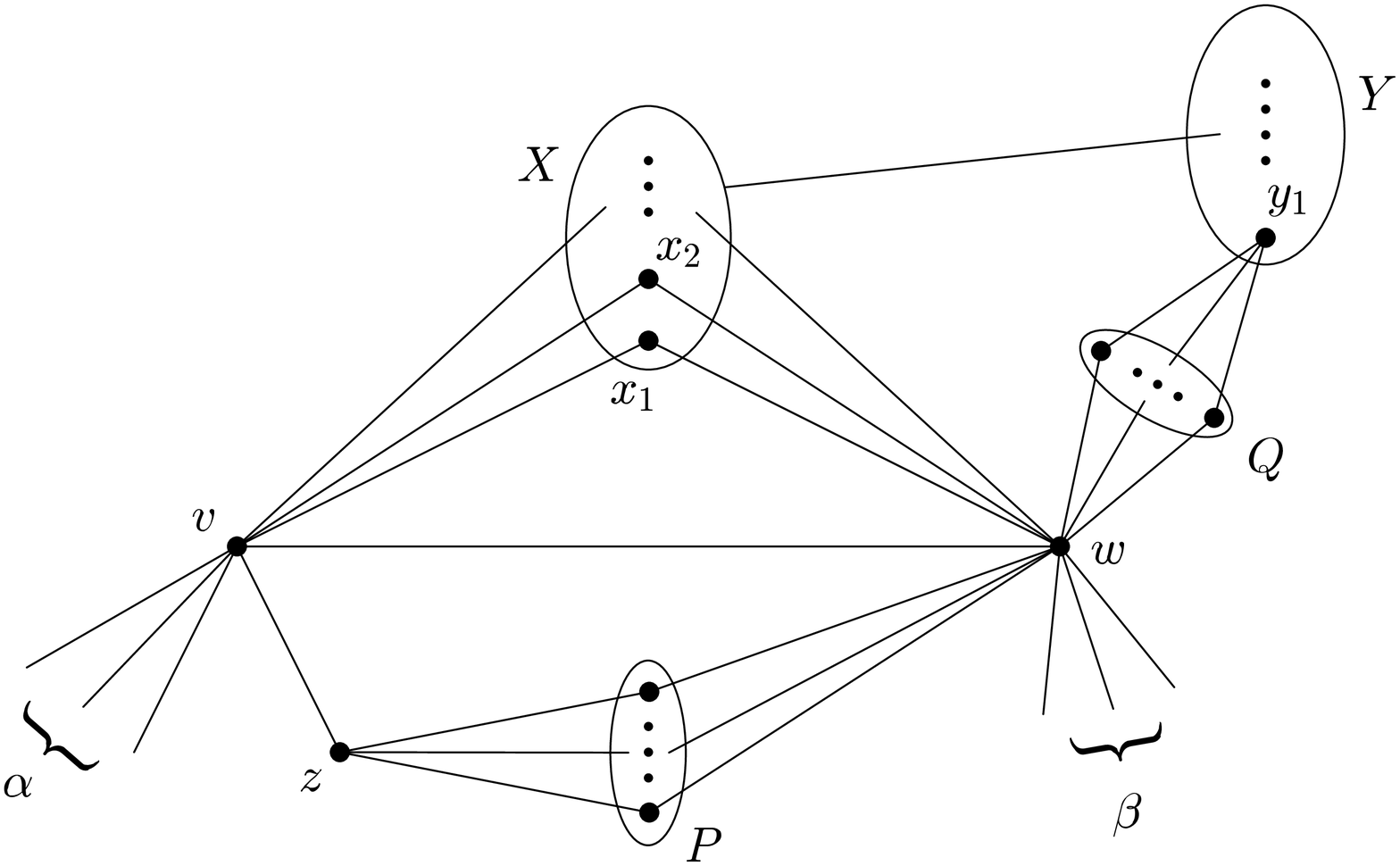}
\caption{A graph whose complement is claw-free, having a triangle, and containing exactly two vertices that are in every triangle ($t=2$).  We assume $|X|\geq2$.  The graph is mock threshold provided that $P\neq\emptyset$ implies $Y=Q=\emptyset$; $Q\neq\emptyset$ implies $P=\emptyset$ and $|Y|=1$; $|Y|\geq2$ implies $P=Q=\emptyset$; $|Y|>2$ implies $|X|=2$.  If $P=\emptyset$, then $z$ may be omitted.}
\label{F:clawfree2}
\end{figure}

\begin{proof}
Consider a mock threshold graph $G$ that is claw-free.  A component that is a tree cannot have a vertex of degree greater than 2; thus, it is a path.  Two components that are not trees would give an induced subgraph that is a disjoint union of cycles, contrary to Lemma \ref{DISJOINTCYCLES}.  Thus, if $G$ is not a forest it has exactly one component, say $G_1$, that is not a tree.  Letting $G_2$ be the 2-core of $G_1$, it is clear that $G_1$ is $G_2$ together with trees that are attached to vertices of $G_2$ by identifying a vertex of the tree with a vertex of $G_2$.  These trees must be paths attached to $G_2$ at an endpoint in order to avoid creating claws in $G_1$.  
If two paths are hung off the same vertex $v$, or if one path is hung off $v$ and there are two non-adjacent neighbors of $v$ in $G_2$, then $G_1$ contains a claw with center $v$; therefore, $G_1$ must be constructed from its 2-core as the theorem states, and the remainder of the proof consists in characterizing $G_2$.  The best way to do so is to characterize its complement, $\overline G_2$, which we call $H$.

As $G_2$ must itself be a claw-free mock threshold graph (Proposition \ref{2-CORE}), we know that $H$ is mock threshold (by Proposition \ref{CLOSEDUNDERCOMPLEMENTATION}), every vertex has at least two non-neighbors (since $\overline H$ is a 2-core), and its complement is claw-free.  The first and last properties imply that $H$ has a removable vertex, and that if $uvw$ is a triangle in $H$, then every vertex is adjacent to at least one of $u, v, w$ (then we say the vertex is \emph{adjacent to the triangle}).  Either $H$ is triangle-free, or it is not.  In the latter case we define $t$ to be the number of vertices of $H$ that belong to every triangle in $H$.  There are consequently five cases in the proof, according as $H$ has no triangles, or it has a triangle and $t=0,1,2,$ or $3$.

\emph{Case 1.  $H$ is triangle-free.}  Then $H$ is bipartite, since it has no long odd cycles.  Bipartite mock threshold graphs were characterized in Theorem \ref{BIPARTITE}.  Every such graph is a possible 2-core complement $H$.  That gives Type \ref{T:clawfreeF} of the theorem.

\emph{Case 2.  $H$ has a triangle and $t=3$.}  That is, $H$ has exactly one triangle, $T$.  Since every other vertex of $H$ must be adjacent to $T$ but cannot be adjacent to more than one vertex of $T$ without forming a second triangle, $H$ consists of a triangle with any number of pendant edges hanging off each vertex.  In order for $\overline H$ to be a 2-core, at least two vertices of $T$ need a pendant edge.  This $H$ is clearly mock threshold; thus, this case is characterized and we have Type \ref{T:clawfree3}.

\emph{Case 3.  $H$ has a triangle and $t=0$.}  This case is impossible.  A vertex not in a triangle must have at least two neighbors, since if it had only one neighbor, that neighbor is not in every triangle.  Thus, $H$ has minimum degree at least 2.  Its complement being a 2-core, it also has minimum codegree at least 2.  Therefore $H$ is not mock threshold, contradicting our hypothesis that $\overline H$ is mock threshold.  That is, Case 3 does not exist.

\emph{Case 4.  $H$ has a triangle and $t=1$.}  
We begin with a lemma that certain graphs, shown in Figure \ref{F:clawfree1}, are mock threshold with claw-free complement.  The lemma includes the case $t=1$ of our theorem but it is less restrictive since it does not require the complement to be a 2-core.

\begin{lemma}\label{L:clawfree1}
A graph is mock threshold and has claw-free complement if it consists of $K_{2,p}$ with $p>0$, whose vertex classes are $X=\{x_1,x_2\}$ and $Z$ of order $p$; also $Y$ of order $s>0$ of which all elements are adjacent to $x_1$ and at least one is adjacent to $x_2$; also $w$ adjacent to every vertex in $X\cup Y$; and any number of pendant edges incident to $w$.
\end{lemma}

\begin{proof}
Since every vertex is adjacent to every triangle, the complement is claw-free.  

The pendant edges on $w$ are removable.  Then the only non-neighbor of $x_1$ is $x_2$, so $x_1$ is removable.  Then all vertices in $Z$ are leaves; deleting them, the remaining graph is clearly mock threshold with dominating vertex $w$.
\end{proof}

Now we assume that $H$ has a triangle, exactly one vertex $w$ belongs to every triangle, and every vertex is adjacent to every triangle.  The fact that only $w$ is in every triangle implies that there are two triangles whose only common part is $w$; let them be $wx_1y_1$ and $wx_2y_2$.  There are at least two vertices not adjacent to $w$, say $z_1$ and $z_2$, and each must be adjacent to one (and only one) of $x_1,y_1$ and one of $x_2,y_2$.  Let us say $z_1$ is adjacent to $x_1,x_2$.  As $H-w$ is triangle-free, it is bipartite; therefore it is impossible for $z_2$ to be adjacent to $x_1$ or $y_2$.  If $z_2$ is adjacent to $y_1,y_2$, then the induced subgraph on $\{w,x_1,x_2,y_1,y_2,z_1,z_2\}$ has no removable vertex, so it is not mock threshold.  It follows that every $z \in Z$ is adjacent to $x_1$ and $x_2$ and not to $y_1$ or $y_2$.

The vertices, other than $w$, of all triangles induce a bipartite graph; call the two vertex classes $X$ and $Y$, chosen so each $z\in Z$ is adjacent to every $x\in X$ (and no $y\in Y$).  Thus, $x_1,x_2\in X$ and $y_1,y_2\in Y$.  Furthermore, since every vertex in $X\cup Y$ is in a triangle, each $y \in Y$ is adjacent to at least one $x\in X$ and vice versa.  It follows that no $y\in Y$ is adjacent to any $z\in Z$, for if it were then $xyz$ would be a triangle for some $x\in X$.

Any vertex in $N(w) - (X\cup Y)$ cannot be adjacent to another neighbor of $w$, since it is not in a triangle.  Therefore, it is a leaf or it is adjacent to at least one $z\in Z$.  Let $B$ be the set of leaf neighbors of $w$ and let $A = N(w) - (X\cup Y\cup B)$; that is, $A$ contains the $w$-neighbors that are not leaves and not in triangles so $V(H) = \{w\} \cup X \cup Y \cup Z \cup A \cup B$.

We show that $|A \cup X| \leq 2$, from which it follows that $A=\emptyset$ and $X=\{x_1,x_2\}$.  Every vertex in $H-B$ has degree at least 2, and the only ones whose codegree can be less than 2 are the $x\in X$.  However, the codegree of $x$ is at least $|A\cup X|-1$; therefore $|A\cup X|\leq2$.  Furthermore, since $x_1$, say, must have codegree 1 and it has the non-neighbor $x_2$, it must be adjacent to all $y\in Y$.  On the other hand, $x_2$ only has to be adjacent to $y_2$.  We have now shown that in Case 4, $H$ is a graph of Type \ref{T:clawfree1}.

\emph{Case 5.  $H$ has a triangle and $t=2$.}  
We begin with another lemma that certain graphs are mock threshold and have complement that is claw-free.  It includes the case $t=2$ but does not require the complement to be a 2-core.

\begin{lemma}\label{L:clawfree5}
A graph of the type in Figure \ref{F:clawfree2} (with $\alpha,\beta\geq0$, with any of $P,Q,X,Y$ possibly empty, and with the restrictions in the caption) is mock threshold and has claw-free complement. 
\end{lemma}

\begin{proof}
All triangles are of the form $vwx_i$ and every vertex is adjacent to $v$ or $w$ or, for vertices in $Y$, to every $x_i\in X$.  Therefore the complement is claw-free.

We prove the graph $F$ is mock threshold.  Leaves are removable, so we may assume $\alpha=\beta=0$ in the figure.  If $|Y|<2$, then $w$ is removable and $F-w$ is mock threshold because it is either a forest or, if $|Y|\geq2$, $K_{2,s}$ with trees attached.  If $|Y|\geq2$, then $P=Q=\emptyset$ so every $x \in X$ is removable; the remaining graph is a forest.
\end{proof}

The lemma proves that Types \ref{T:clawfree20}--\ref{T:clawfree22} are claw-free mock threshold graphs.  One can verify by inspection that all codegrees are at least 2.  We have to prove that $H$ is of one of those types. 

In order to have exactly two vertices, say $v$ and $w$, in every triangle, $H$ must have at least two vertices adjacent to both and those two cannot be adjacent.  Let $X=\{x_1,\ldots,x_r\}
$ be the set of vertices adjacent to $v$ and $w$; then $r\geq2$.  Let $Y=\{y_1,\ldots,y_s\}$ (with $s\geq0$) be the set of vertices that are non-adjacent to both $v$ and $w$; then every $y_j$ must be adjacent to every $x_i$ for $\overline H$ to be claw-free and consequently no two $y_j$'s can be adjacent, so we have an induced subgraph $K_{r,s}$.  

Any remaining vertex of $H$ must be adjacent to $v$ or $w$ but not both.  Let $A=N(v)-X$ and $B=N(w)-X$; then $V(H)=\{v,w\}\cup X \cup Y \cup A \cup B$.  
The edges involving vertices $a\in A$ or $b\in B$ are limited, as we show in the subcases.  

An edge $bx_i$ would form a triangle without $w$; thus it cannot exist, nor can $ax_i$.  

By Proposition \ref{BIPARTITE}, $r=2$ or $s\leq2$.  We treat cases according to the value of $s$.  We first prove each case falls under Lemma \ref{L:clawfree5}, hence is mock threshold and claw-free; then we verify when the complement is a 2-core, i.e., all codegrees $\bard$ are at least 2.

When $s=0$, an edge $ab$ is possible, but suppose there are two such edges that are not adjacent, $ab$ and $a'b'$.  Then the induced subgraph $H{:}\{v,w,x_1,x_2,a,a',b,b'\}$ has no removable vertex; thus it is not mock threshold.  This applies whether or not $H$ has edges $ab'$ and $a'b$.  It follows that the edges between $A$ and $B$, if any, constitute a star, which we may assume is part of an induced subgraph $K_{2,p}$ with vertex sets $\{z,w\}$ and $P$.  

In order for $\overline H$ to be a 2-core, $H$ must have no vertex with codegree 0 or 1.  To get $\bard(v) \geq 2$, $w$ must have at least 2 neighbors in $B$.  
If $p=0$ these are leaves.  Thus $\beta$ (and $\alpha$) are at least 2.   Then every $\bard(x_i)\geq5$ so the degree condition for being a 2-core complement is satisfied.  This gives us the graph of Type \ref{T:clawfree20}.  
If $p=|P|>0$, then $v$ has a neighbor $z$ that is not a leaf so the number $\alpha$ of pendant edges at $v$ only needs to be positive to make $\bard(w)\geq2$.  Since $v$ has all of $P$ as non-neighbors, only when $p=1$ is it necessary for $w$ to have a pendant edge; that is, for $\beta$ to be positive.  Thus we have Type \ref{T:clawfree20p}.

When $s>0$, no edge $ab$ is possible, since if there were such an edge, the induced subgraph $H{:}\{v,w,a,b,x_1,x_2,y_j\}$ would have no removable vertex.

In case $s=1$ there cannot simultaneously be edges $ay_1$ and $by_1$ since that would make $H{:}\{v,w,a,b,x_1,x_2,y_1\}$ have no removable vertex.  If there is no edge of either type, we need at least one pendant edge at each of $v$ and $w$ to make $\bard(w)\geq2$ and $\bard(v)\geq2$, respectively, in order to guarantee that $\overline H$ is a 2-core; that gives Type \ref{T:clawfree21}.  If there is an edge $by_1$, there may be several of them; the $q$ vertices of $B$ in such edges along with $w$ and $y$ induce a $K_{2,q}$ with one vertex class $\{w,y_1\}$.  Then we need at least one pendant edge at $v$ to make $\bard(w)\geq2$.  This gives Type \ref{T:clawfree21q}.

If $s\geq2$, suppose there were an edge $by_1$; then $H{:}\{v,w,x_1,x_2,b,y_1,y_2\}$ has no removable vertex.  Therefore $by_j$ and similarly $ay_j$ cannot exist.  The codegrees are $\bard(y_j)\geq s+1>2$, $\bard(x_i)=r-1+\alpha+\beta$, and $\bard(v),\bard(w)\geq s\geq2$.  To ensure that $\overline H$ is a 2-core, $r-1+\alpha+\beta$ must be at least 2, which implies $\alpha+\beta \geq 1$ if $r=2$ while if $r>2$ there is no restriction on $\alpha$ and $\beta$.  Thus we have Type \ref{T:clawfree22}.
\end{proof}

\section {Line graphs}
 
The line graph of $G$, denoted by $L(G)$, has the edges of $G$ as its vertices, with two vertices adjacent if they are adjacent as edges in $G$. Line graphs are important in graph theory for various reasons, one of the most important being that they translate questions about edges into questions about vertices and vice versa. The class of line graphs is closed under taking induced subgraphs and there is a beautiful characterization of this class in terms of its forbidden induced subgraphs \cite{LWB}. 
 
 In this section we characterize graphs whose line graphs are mock threshold.  First, we review the facts about threshold line graphs (for which we did not find a reference).
 
 \begin{proposition}\label{P:lgthresh}
A graph is a threshold line graph if and only if it is complete or has a vertex of degree $1$ or $2$ whose deletion results in a complete graph.
\end{proposition}

\begin{proposition}\label{P:Glgthresh}
The line graph $L(G)$ is threshold if and only if every component of $G$ is an isolated vertex or edge, except possibly one, which is a star, optionally with one added edge.
\end{proposition}

Another way to describe the non-isolated component (if any) of $G$ is as a connected subgraph of a triangle with any number of pendant edges at one vertex.

\begin{proof}[Proof Sketch]
Proposition \ref{P:lgthresh} follows immediately from the characterization of claw-free threshold graphs in Proposition \ref{P:clawfreethresh}.  Then it is easy to deduce Proposition \ref{P:Glgthresh} by using Whitney's theorem that the root graph of a connected line graph is unique with the exception that $L(K_3)=L(K_{1,3})$.
\end{proof}

 We start the treatment of mock threshold line graphs with a well-known fact.

 \begin{lemma}\label{LINEGRAPHLEMMA}
  If $H$ is a subgraph of $G$, then $L(H)$ is an induced subgraph of $L(G)$. 
 \end{lemma}
 
 As a consequence of this, the class of all graphs whose line graphs are mock threshold is closed under taking subgraphs and hence can be characterized by a list of forbidden subgraphs. This class also has a nice structural characterization. These two characterizations are combined in the following theorem.

 \begin{theorem}[Forbidden Subgraphs for Mock Threshold Line Graphs]\label{LGrootgraphs}
  Let $G$ be a graph. The following statements are equivalent.
  \begin{enumerate}[{\rm(1)}]
   \item $L(G)$ is mock threshold.
   \item $G$ contains no cycle of length at least $5$ and none of the twelve graphs shown in Figure \ref{Forbidden12Graphs}.
  
   \item $G$ is one of the following:
   \begin{enumerate}[{\rm(a)}]
   \item A linear forest (every component is a path).
   \item Only one component is not a path and it is obtained by iteratively adding a pendant edge to a leaf in one of the following:
   \begin{enumerate}[Type 1.]
   \item A star.
   \item A star plus two pendant edges on a leaf.
   \item $C_3$ with pendant edges on one vertex and at most one more pendant edge on one of the other two vertices.
   \item $C_4$ with pendant edges on one vertex and at most one pendant edge on one adjacent vertex.
   \item $K_4 -e $ with pendant edges on a trivalent vertex and at most one more pendant edge on one divalent vertex.
   \item $K_4 -e $ with two pendant edges on a divalent vertex.
   \item $K_4$ with at most two pendant edges on one vertex.
   \item Two triangles at a vertex with pendant edges on the vertex.
   \end{enumerate}
   
  \end{enumerate}
  
  \end{enumerate}
\end{theorem}
 
\begin{figure}
\includegraphics[scale=.3]
{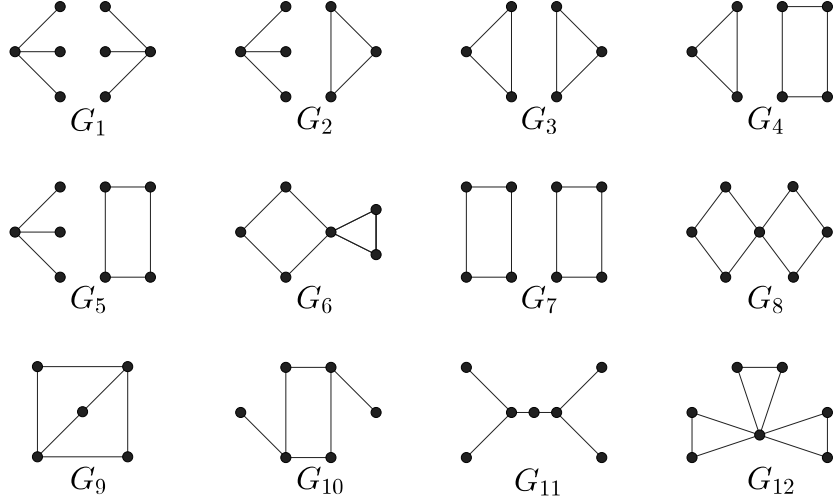}
\caption{Cycles of length at least $5$ and these 12 graphs are the minimal forbidden subgraphs for the class of graphs whose line graphs are mock threshold.}
\label{Forbidden12Graphs}
\end{figure}

\begin{proof} We will prove (1) $\Rightarrow$ (2) $\Rightarrow$ (3) $\Rightarrow (1)$. 
\bigskip
   
(1) $\Rightarrow$ (2).  The line graph of each of the graphs in Figure \ref{Forbidden12Graphs} is non-mock threshold. One way to see this is to observe that every vertex in the the line graph of each one of the graphs in the list has degree and codegree greater than $1$ and then use Lemma \ref{MINDEGREEMAXDEGREE}.   This together with Lemma \ref{LINEGRAPHLEMMA}  and Proposition \ref{CLOSEDUNDERINDUCEDSUBGRAPHS} completes the proof.  
\bigskip

(2) $\Rightarrow$ (3). Assume $G$ is a graph that contains none of the graphs in the list as a subgraph. We will show that $L(G)$ is mock threshold. If $G$ has two non-path components, then it would contain one of the forbidden subgraphs: $C_n (n \geq 5), G_1, G_2 , G_3, G_4, G_5, G_7$. Hence $G$ contains at most one component that is not a path. If $G$ has no non-path component, then we are done. Hence we will focus on the case where $G$ has exactly one non-path component. Also, we will reduce $G$, by which we mean repeatedly deleting leaves adjacent to vertices of degree $2$. 

Note that $G$ contains only cycles of length $3$ or $4$. 

\emph{Case 1. $G$ has no cycles.} In this case $G$ is a tree. Since $G$ contains neither $G_1$ nor $G_{11}$, any two vertices of degree at least three must be adjacent. Since $G$ has no triangles this implies $G$ has at most two vertices of degree at least $3$.  If $G$ has only one vertex of degree greater than $2$, then $G$ is a star (Type 1). If $G$ has two vertices that have degree greater than $2$, then one of them has to have degree $3$ (since $G$ does not contain $G_1$). This gives Type 2.   

\emph{Case 2. $G$ has a $3$-cycle but no $4$-cycle.} Let $X = \{v_a,v_b,v_c\}$ form a triangle in $G$. Since $G$ has no bigger cycles, every vertex in $V(G)-X$ has at most one neighbor in $X$. Since $G$ is reduced, every vertex in $V(G)-X$ is adjacent to at least one vertex in $X$. Observe that $G-X$ can have at most one edge (otherwise it would contain a cycle of length greater than $3$). If $G-X$ has exactly one edge, then $G$ is of Type 8. Otherwise, $G$ is the triangle on $X$ together with pendant edges (forming a star) at each of $a,b,c$. If two of these stars have size more than one, then $G$ would contain $G_{11}$. Hence $G$ is of Type 3.    

\emph{Case 3. $G$ has an induced $4$-cycle.} Let $X = \{v_a,v_b,v_c,v_d\}$ induce a $4$-cycle in $G$. Since $G$ contains neither $C_5$ nor $G_9$, every vertex in $V(G)-X$ has at most one neighbor in $X$. Since $G$ is reduced, every vertex in $V(G)-X$ is adjacent to at least one vertex in $X$. Hence every vertex in $V(G)-X$ is adjacent to exactly one vertex in $X$. Since $G$ contains neither $G_6$ nor cycles of length greater at least $5$, $G-X$ is edgeless. 

Hence $G$ is the $4$-cycle on $X$ together with pendant edges (forming a star) at each of its vertices. Since $G_{10}$ is not a subgraph of $G$, the stars can be at two adjacent vertices only. If both stars have size more than one, then $G$ would contain $G_1$. Hence $G$ is of Type 4.

\emph{Case 4. $G$ has an induced $K_4 -e $.} Let $X = \{v_a,v_b,v_c,v_d\}$ induce a $K_4-e$ in $G$ where $v_bv_d$ is the missing edge. Arguing exactly as in the previous case, we see that $G-X$ is edgeless.  Hence $G$ is the $K_4-e$ on $X$ together with pendant edges (forming a star) at each of its vertices. Since $G$ does not contain $G_2$, the stars at  $v_b$ and $v_d$ can have at most two edges. If one of them has two edges, then there can be no other edges since $G$ contains neither $G_{10}$ nor $G_{11}$. Hence $G$ is  of Type 6. Otherwise $G$ is of type 5.

\emph{Case 5. $G$ contains $K_4$.}  If two of these vertices have degree more than $3$, then $G$ would contain $C_5$ or $G_{10}$. Hence only one of these vertices (say $v_a$) has degree more than $3$. Also the neighbors of $v_a$ outside the $4$-clique must form an independent set for $G$ does not contain $G_3$. Since $G$ does not contain $G_2$, the degree of $v_a$ must be at most $5$. Hence $G$ is of Type 7. 

This concludes the proof that (2) $\Rightarrow$ (3). 
\bigskip

(3) $\Rightarrow$ (1). It is straightforward to check that the line graph of a graph belonging to any of the eight types is mock threshold. We omit the details. 
\end{proof}

\begin{figure}
\includegraphics[scale=.175]
{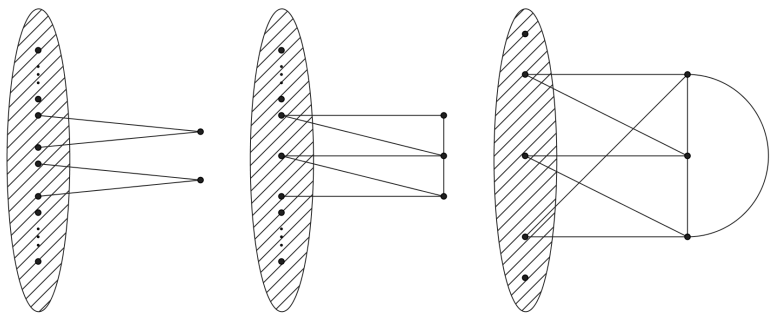}
\caption{Mock threshold line graphs. The shaded part denotes a clique. In the first two the clique can be of any size whereas in the third the clique has $5$ vertices.}
\label{mockline2}
\end{figure}

\begin{corollary}[Mock Threshold Line Graphs]
A graph is a mock threshold line graph if and only if it is a linear forest or a graph with exactly one non-path component which can be obtained from a connected induced subgraph of one of the three graphs in Figure \ref{mockline2}  by repeatedly adding pendant edges to leaves.
\end{corollary}

\section{Miscellaneous topics}

We would like to characterize the mock threshold graphs with other standard graph properties.  
We can characterize planarity and outerplanarity; that will appear separately.
We have minor or partial results on other properties, that we address in this section.

\begin{subsection}
{Chordality}\

Threshold graphs are chordal, as can easily be seen from their characterization by forbidden induced subgraphs (Theorem \ref{ThresholdFISC}). But mock threshold graphs need not be chordal. We characterize those that are chordal. 

\begin{proposition}
A mock threshold graph is chordal if and only if it has an MT-ordering $v_1, \ldots, v_n$ such that for every $i$ ($1 \leq i \leq n$), if the degree of $v_i$ in $G{:}\{v_1, \ldots, v_i\}$ is $i-2$, the unique non-neighbor of $v_i$ in $\{v_1, \ldots, v_i\}$ is simplicial in $G{:}\{v_1, \ldots, v_i\}$. 
\end{proposition}

\begin{proof}
Let $G$ be a mock threshold graph with an MT-ordering $v_1, \ldots, v_n$ such that for every $i$ ($1 \leq i \leq n$) such that the degree of $v_i$ in $G{:}\{v_1, \ldots, v_i\}$ is $i-2$, the unique non-neighbor of $v_i$ in $G{:}\{v_1, \ldots, v_i\}$ is simplicial. 
Suppose $G$ contains an induced $k$-cycle ($k \geq 4$), say $C$. Let $i$ be the largest index of a vertex in $C$. Since $v_i$ has degree $2$ in $C$, $\{v_1, \ldots, v_n\}$ is an MT-ordering, and $C$ is induced in $G$, $k = 4$. This means the unique non-neighbor of $v_i$ in $\{v_1, \ldots, v_i\}$ is not simplicial, a contradiction. Hence $G$ does not contain an induced $k$-cycle for some $k \geq 4$, and hence is chordal. 
 
For the converse, we prove the contrapositive. Let $G$ be a mock threshold graph with an MT-ordering $v_1, \ldots, v_n$. Let $v_i$ be a vertex that is non-adjacent to exactly one vertex in $\{v_1, \ldots, v_i\}$, say $v_k$. Suppose $v_k$ is not simplicial in $G{:} \{v_1,\ldots,v_i\}$. Then $v_k$ has neighbors $v_a$ and $v_b$ that are non-adjacent. Now $\{v_i,v_j,v_a,v_b\}$ induces a $4$-cycle in $G$, and hence is not chordal. 
\end{proof}

\end{subsection}

\begin{subsection}
 {Evenness}\

A graph is even if all its vertices have even degree and Eulerian if it is even and connected. We cannot characterize even or Eulerian mock threshold graphs, but we will show that every mock threshold graph with $n$ vertices is an induced subgraph of a mock threshold even graph with at most $2n$ vertices.  But first, for comparison, we state an easy characterization of even and Eulerian threshold graphs.

\begin{proposition}\label{P:even}
A threshold graph is even if and only if, in a threshold ordering, the length of each consecutive string of dominating vertices is even.  It is Eulerian if and only if it is even and the last vertex added is dominating.
\end{proposition}

Let $G$ be a graph. Let $v_1, \ldots ,v_k$ be the vertices with even degree. Let $w_1, \ldots, w_p$ be the vertices with odd degree. (Thus $p$ is even.)  Let $G'$ be obtained from $G$ by adding vertices $x_1, \ldots, x_p$ and $x_i$ is adjacent to all vertices in $G'$ except $w_i$.
Since $G'$ is obtained from $G$ by adding near-dominating vertices, 

\begin{lemma}\label{MTGPreservation}
If $G$ is mock threshold, then so is $G'$.
\end{lemma}

The following fact is straightforward to verify. 

\begin{lemma}\label{Evenness}
If $G$ has an odd number of vertices, then $G'$ is even. 
\end{lemma}

We will use these two lemmas to prove an evenness property.

\begin{proposition} 
Every mock threshold graph $G$ with $n$ vertices is an induced subgraph of a mock threshold even graph with at most $2n$ vertices.
\end{proposition}
\begin{proof}
Let $G$ be a mock threshold graph on $n$ vertices. If $G$ is already even, we are done. If $G$ has an odd number of vertices, then $G'$ is mock threshold (by Lemma \ref{MTGPreservation}) and is even (by Lemma \ref{Evenness}). If $G$ has an even number of vertices, then $H'$ is mock threshold and even, where $H$ is obtained from $G$ by adding a dominating vertex. 
\end{proof}
\end{subsection}

\begin{subsection}
{A big clique or a big stable set}\
 
It is known that  every graph on $n$ vertices contains a clique or stable set of size at least $\frac{1}{2} \log_2 {n}$.
It is also known that every perfect graph on $n$ vertices contains a clique or stable set of size at least $\sqrt{n}$. 
Since mock threshold graphs are more structured, a stronger conclusion holds for them. 
 
\begin{proposition}
Every mock threshold graph on $n$ vertices contains a clique or stable set of size at least $\frac{n}{4}$. 
\end{proposition}
\begin{proof}
Let $G$ be a mock threshold graph on $n$ vertices. Consider one of its MT-orderings. Color a vertex red if it is of dominating or near-dominating type, and blue otherwise. By the pigeonhole principle there must be at least $\frac{n}{2}$  blue vertices (passing to the complement, if necessary). Call the set of blue vertices $B$. Then $G{:}B$ is a forest, which is bipartite and hence has a stable set of size at least $\frac{1}{2}|B| \geq \frac{1}{2} \cdot \frac{n}{2} = \frac{n}{4}$.   
\end{proof}
 
For comparison, a threshold graph has a clique or stable set of size at least $\frac{n}{2}$. 

\end{subsection}

\begin{subsection}
{Well-quasi-ordering}\

A quasi-order is a pair $(Q, \leq)$, where $Q$ is a set and $\leq$ is a reflexive and transitive relation on $Q$. An infinite sequence $q_1,q_2, \ldots $ in $(Q, \leq)$ is \textit{good} if there exist $i < j$ such that $q_i \leq q_j$. A quasi-order $(Q, \leq)$ is a well-quasi-order (WQO) if every infinite sequence is good.  Some graph classes are well-quasi-ordered, others are not; and it can depend upon the chosen ordering.  
Threshold graphs are well-quasi-ordered under both the induced subgraph relation and the subgraph relation, but mock threshold graphs cannot be added to the list of well-quasi-ordered graph classes because they include trees.

\begin{proposition}
Mock threshold graphs are not well-quasi-ordered under either the subgraph relation or the induced subgraph relation.  
\end{proposition}
\begin{proof}
Trees are not well-quasi-ordered under either relation. 
Take a path of length $i$ and attach two pendant edges at each of its ends. Call this tree $T_i$. Then $T_1,T_2, \ldots$ is a sequence that is not good since no $T_i$ is a subgraph of $T_j$ for  $i \not = j$. We are done by Proposition \ref{FORESTSAREMOCKTHRESHOLD} and the fact that if $(Q, \leq)$ is a WQO, then any subset of $Q$ is a WQO with respect to $\leq$.  
\end{proof}

\end{subsection}

\section{Algorithmic aspects}\label{algorithm}
 
\subsection{Recognition algorithm}\

There is an easy algorithm for recognizing both mock threshold graphs and non-mock threshold graphs. 
  
\bigskip
\noindent 
\emph{Input}: A graph $G$. \\
\emph{Algorithm}: Check whether there is a removable vertex. If not, declare that the graph is not mock threshold. If there is, delete that vertex and repeat the procedure.  
If all vertices are removed, declare that the graph is mock threshold.
\bigskip
  
The algorithm works because of Lemmas \ref{MINDEGREEMAXDEGREE} and \ref{REMOVABLE}.

\subsection{Chromatic and clique numbers}\

Determining whether a graph has chromatic number $3$ is NP-complete.  Since a mock threshold graph is perfect, its chromatic number and clique number coincide. The special structure of mock threshold graphs makes it possible to determine these numbers in polynomial time.
\bigskip

\noindent
\emph{Algorithm to compute the clique number of a mock threshold graph.}

Step 1: Determine an MT-ordering for the given graph $G$. 

Step 2: We look at vertices one-by-one, starting with the last vertex in the MT-ordering. If the scanned vertex is dominating or almost dominating, delete the vertex and its non-neighbor (if there is one) and increase the clique number by $1$. If the scanned vertex is pendant or isolated, just delete it. 
\bigskip

Since mock threshold graphs are weakly chordal, any efficient algorithm that applies to weakly chordal graphs also works for mock threshold graphs.

\subsection{Bandwidth}\

Despite the ease of the recognition and chromatic/clique number problems, there is a difference in complexity between threshold and mock threshold graphs.  
The bandwidth problem illustrates it.  Bandwidth is linear-time solvable (Yan, Chen, and Chang \cite{YCC}) for quasi-threshold graphs, which include threshold graphs, but NP-complete for mock threshold graphs since it is NP-complete for trees \cite{GGJK}.

\section{Open problems}

We mentioned the problem of characterizing even mock threshold graphs.  Here are more open problems that are worthwhile but may be difficult.

\subsection{Degree sequences}\
 
A mock threshold graph and a non-mock threshold  graph can have the same degree sequence. We give an example. Let $G_1$ be the graph with two components: one being a $5$-cycle with a chord, and the other $K_2$. Let $G_2$ be the graph consisting of a $5$-cycle together with two pendant edges at two adjacent vertices on the cycle. The degree sequence of both $G_1$ and $G_2$ is $(3,3,2,2,2,1,1)$, but $G_1$ is mock threshold whereas $G_2$ is not. 
This is in stark contrast with threshold graphs, where the graph is determined by the degree sequence. 

\begin{problem}\label{MTsequences}
{\rm(a)  Characterize graphic sequences all of whose realizations are mock threshold.}

{\rm(b)  Characterize graphic sequences none of whose realizations are mock threshold.}
\end{problem}

Amongst the solutions to (b) will be the graphic sequences in which all degrees lie between $2$ and $|V(G)|-3$; but those are not all.

 \subsection{Splitness}\
 
 Let $\SplitMT$ denote the set of all graphs that are both split and mock threshold.  
 \begin{problem}\label{SPLITPROBLEM}
 Characterize graphs that are both split and mock threshold. More specifically, determine  $\Forb(\SplitMT)$.
 \end{problem}

\subsection{Chromatic and Tutte polynomials}\

The chromatic polynomial of a threshold graph is easy to compute because the graph is chordal.  Mock threshold graphs, however, are not chordal so the chromatic polynomial is not readily calculated.

\begin{problem}
Can we say anything about the chromatic or Tutte polynomial of a mock threshold graph? 
\end{problem}

We do not know whether anything is known about the Tutte polynomial of even a threshold graph.  

\subsection{Hamiltonicity}\

Harary and Peled \cite{HP} characterized Hamiltonian threshold graphs.   

\begin{problem}\label{HAMPROBLEM}
Give  a similar characterization for Hamiltonian mock threshold graphs. 
\end{problem}

\subsection
{Further generalization}\

Let $k$ be a non-negative integer. Let $\mathcal{G}_k$ denote the class of graphs that can be constructed from $K_1$ by repeatedly adding a vertex with at most $k$ neighbors or at most $k$ non-neighbors; call these graphs \emph{$k$-mock threshold}.  Then we have a nested sequence: $\mathcal{G}_0 \subset \mathcal{G}_1 \subset \mathcal{G}_2 \subset \cdots .$  Each containment is proper; for example, a $k$-regular graph on $2k+2$ vertices is in $\mathcal{G}_k$ but not $\mathcal{G}_{k-1}$.  Clearly, $\bigcup \mathcal{G}_k$ is the set of all graphs.  Each $\mathcal{G}_k$ is closed under taking complements and induced subgraphs (though certainly not subgraphs).

A general goal would be to understand $\mathcal{G}_k$ and a specific natural problem will be to determine $\Forb(\mathcal{G}_k)$.  The case $k=0$ is that of threshold graphs, which is known, and the case $k=1$ is Theorem \ref{FORBIDDEN}.  $\Forb(\mathcal{G}_k)$ consists of graphs $G$ such that $G \notin \mathcal{G}_k$ and $G - v \in \mathcal{G}_k$ for every $v \in V(G)$.  It contains every connected $(k+1)$-regular graph on at least $2k+3$ vertices.  Presumably there are also sporadic members, as with $\Forb(\mathcal{G}_1)$; an open question is whether their number is finite.

We go out of the realm of perfect graphs when $k \geq 2$.  In fact, all the minimal imperfect graphs, which are $C_n$ and $\overline C_n$ for odd $n\geq5$, are 2-mock threshold.

\begin{acknowledgements}
We are grateful to Jeff Nye for helping us with the Mathematica program for determining the minimal non-MTGs with at most $11$ vertices, which was crucial for discovering that the largest sporadic forbidden graphs have no more than $10$ vertices.
\end{acknowledgements}


\begin{figure}[hb]
\includegraphics
[scale=.575]
{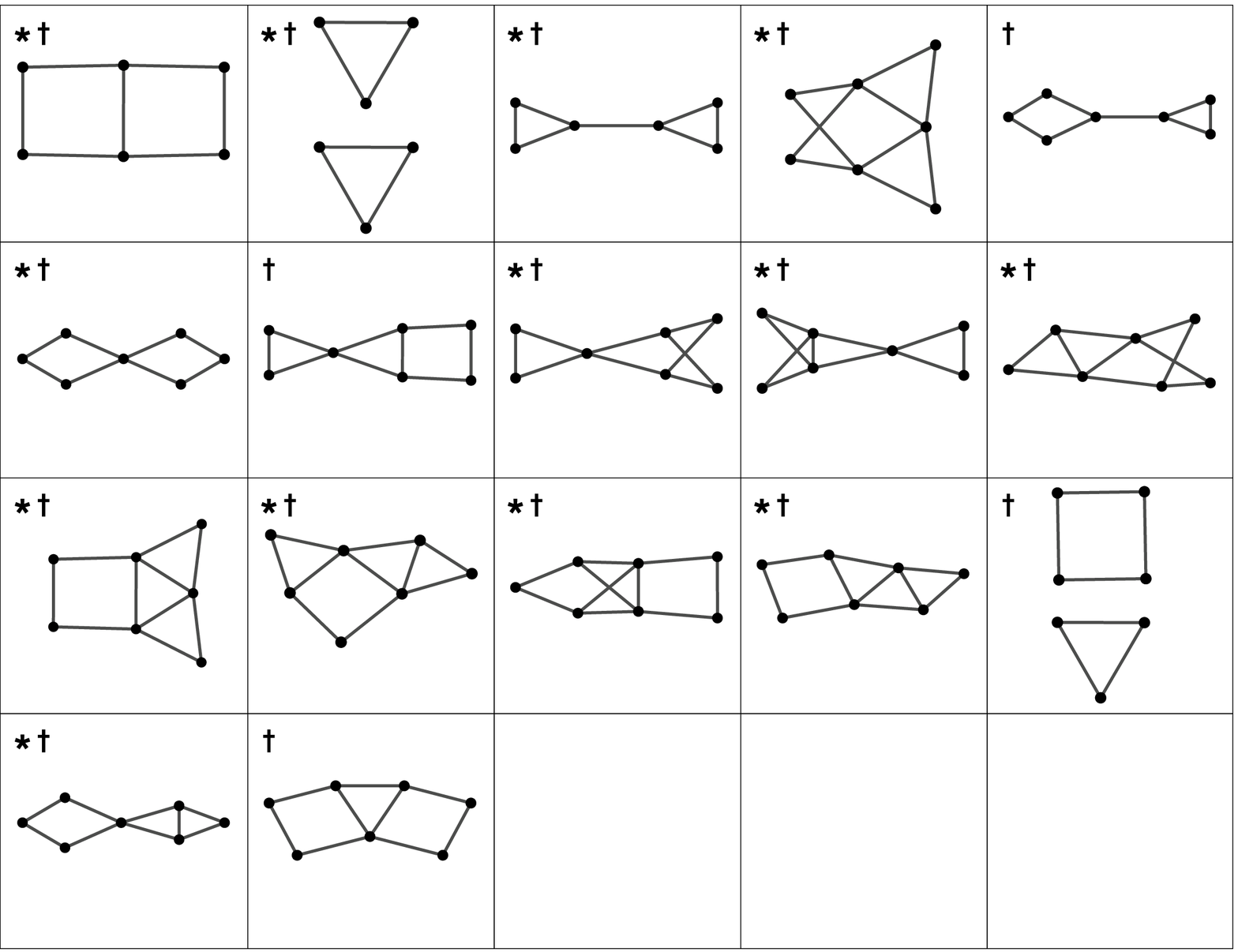}
\caption{Members of $\Forb(\MT)$ with at most $7$ vertices, except cycles and cycle complements. There are $34$ of them: the $17$ shown and their complements. \newline $*$ marks a graph that is minimally non--mock-threshold under contraction. \newline \dag\ denotes a graph whose complement is minimally non--mock-threshold under contraction.}
\label{ForbiddenMTGraphs7Vertices}
\end{figure}

\begin{figure}
\includegraphics
[scale=.5]
{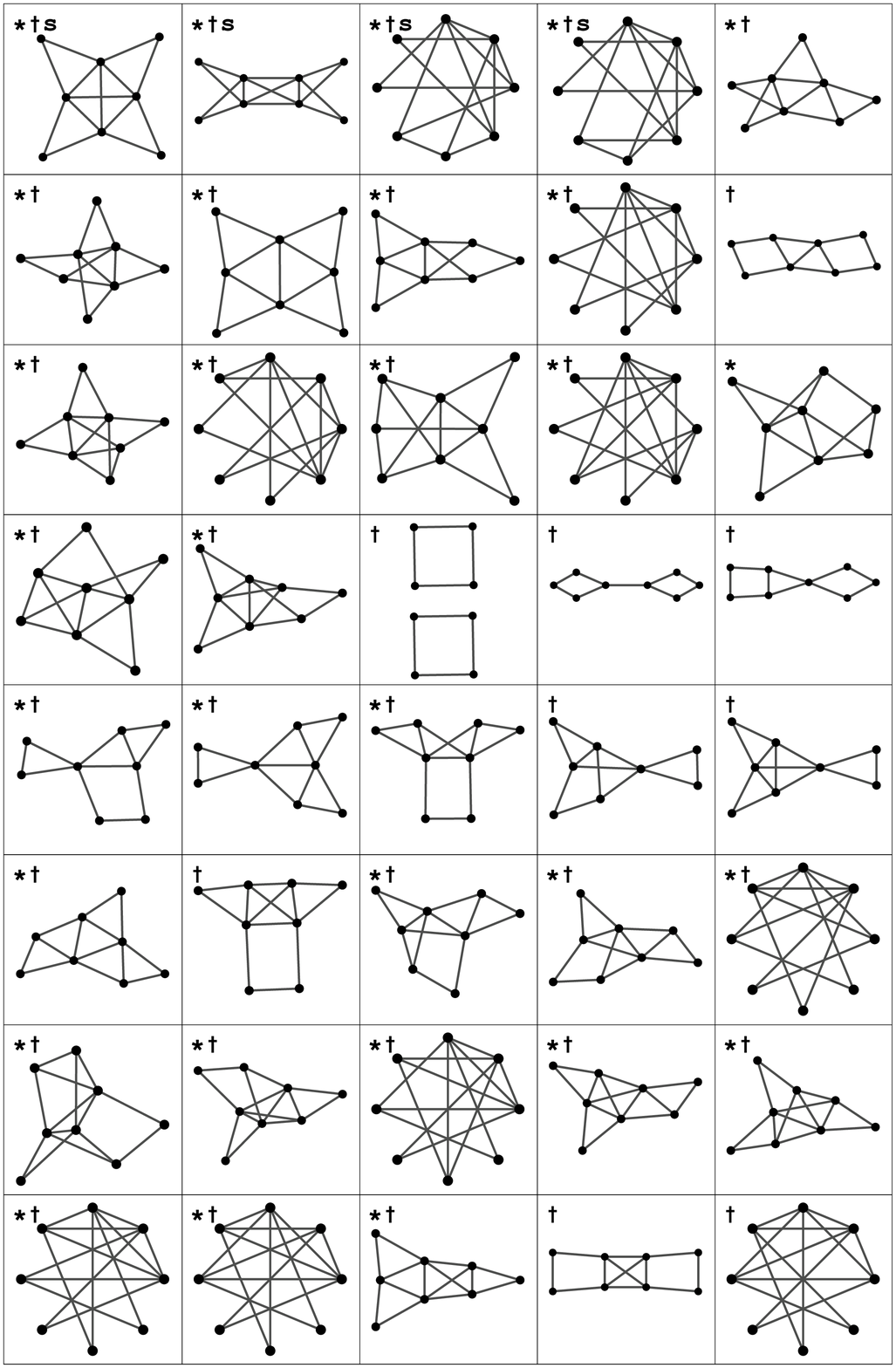}
\caption{Members of $\Forb(\MT)$ with $8$ vertices, other than $C_8$ and $\overline C_8$; the figure shows one member of each complementary pair. \newline s denotes a self-complementary graph (four in the top row; there are five in all, including $C_5$).}
\label{ForbiddenMTGraphs8Vertices}
\end{figure}

\begin{figure}
\includegraphics
[scale=.5]
{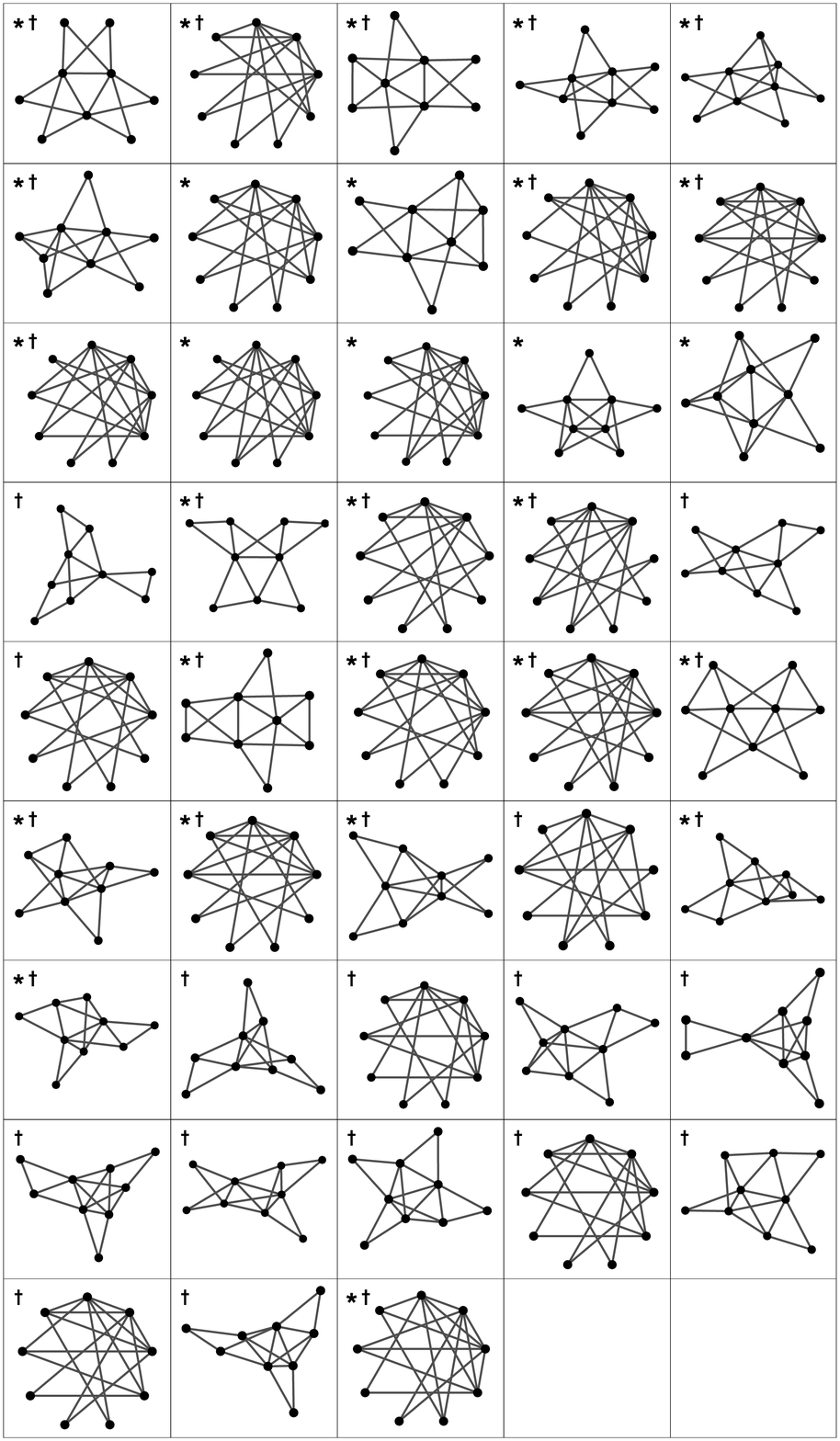}
\caption{One member of each complementary pair of graphs in $\Forb(\MT)$ with $9$ vertices, other than $C_9$ and $\overline C_9$ (first half).}
\label{ForbiddenMTGraphs9VerticesPart1}
\end{figure}

\begin{figure}
\includegraphics
[scale=.5]
{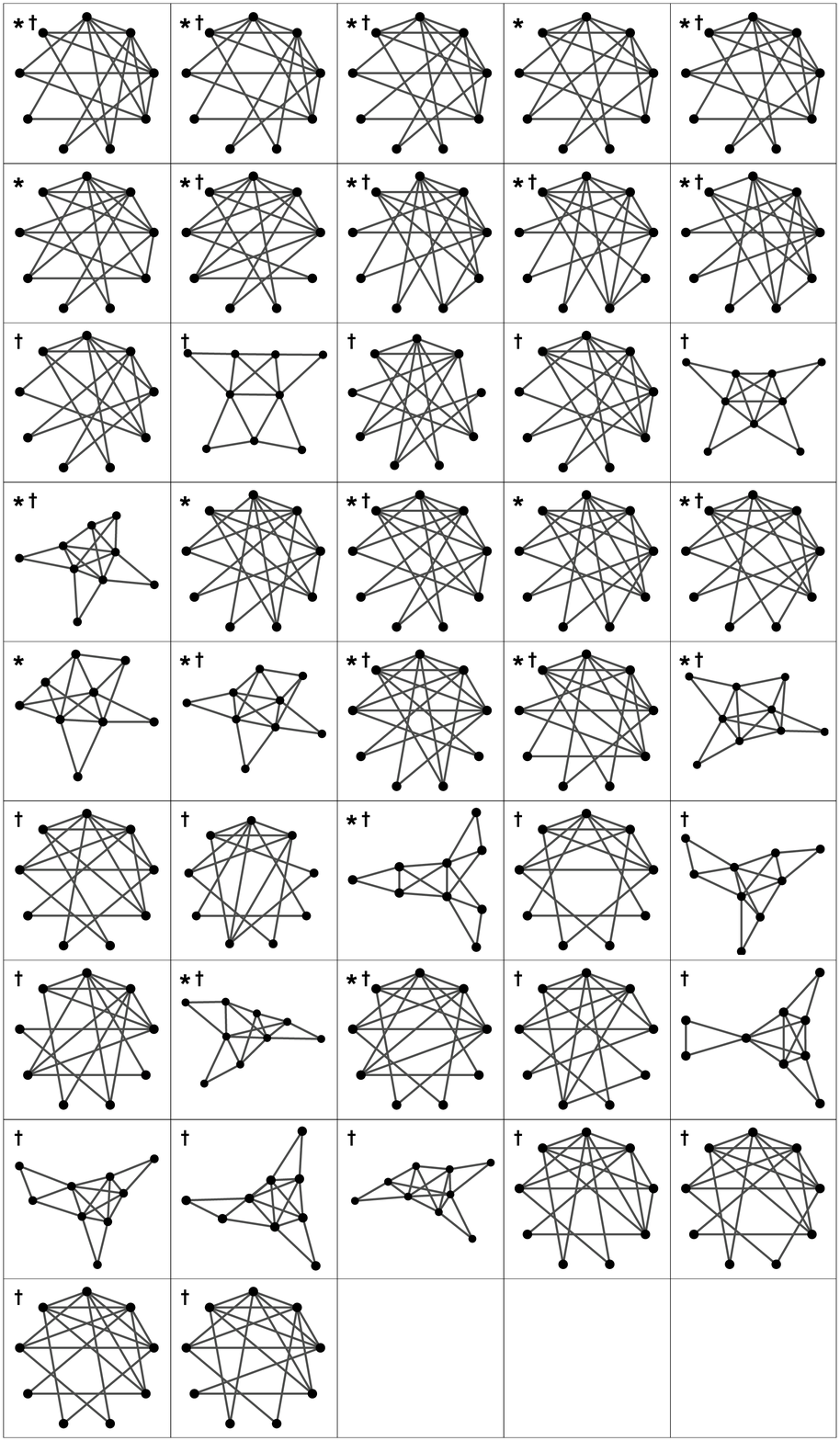}
\caption{One member of each complementary pair of graphs in $\Forb(\MT)$ with $9$ vertices (second half).}
\label{ForbiddenMTGraphs9VerticesPart2}
\end{figure}

\begin{figure}
\includegraphics
[scale=.5]
{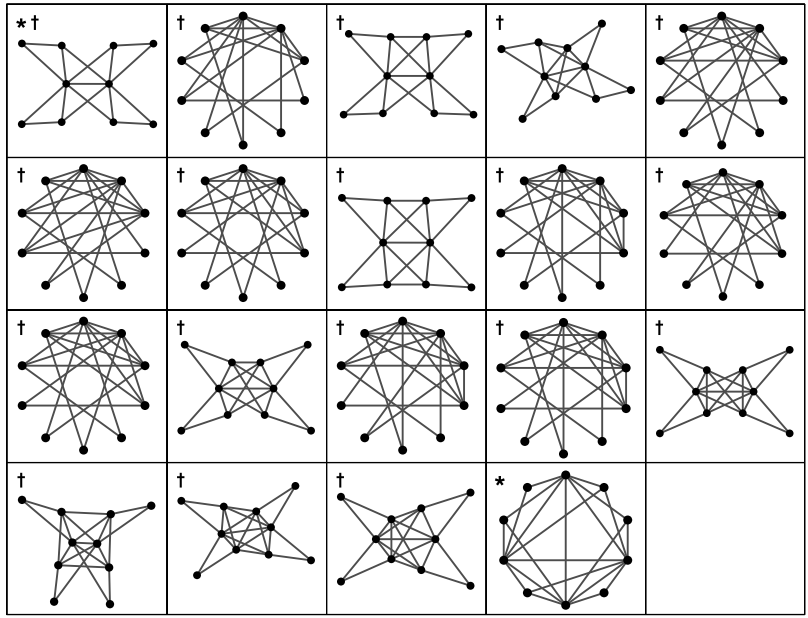}
\caption{Members of $\Forb(\MT)$ with $10$ vertices, other than $C_{10}$ and $\overline C_{10}$. There are $38$ of them: the $19$ shown and their complements.}
\label{ForbiddenMTGraphs10Vertices}
\end{figure}

\end{document}